\documentclass[a4paper,11pt]{amsart}

\usepackage[utf8]{inputenc}
\usepackage[T1]{fontenc}

\usepackage{xcolor}

\usepackage{amsmath,amsfonts,amssymb,amsxtra,setspace,xspace,graphicx,lmodern,psfrag,color,latexsym,bbm,comment,mathtools,amsthm,enumerate,enumitem,pifont,mathrsfs,caption,microtype,dsfont,bm}

\usepackage[colorlinks=true, linkcolor=black, citecolor=red]{hyperref}

\makeatletter
\@namedef{subjclassname@2020}{\textup{2020} Mathematics Subject Classification}
\makeatother

\usepackage{cancel}

\usepackage[a4paper]{geometry}

\makeatletter
 \def\@textbottom{\vskip \z@ \@plus 17pt}
 \let\@texttop\relax
\makeatother
\usepackage[acronym,nonumberlist,nogroupskip]{glossaries}
\usepackage[capitalize]{cleveref}
\usepackage{autonum} %
\newtheorem{thm}{Theorem}[section]
\newtheorem{cor}[thm]{Corollary}
\newtheorem{lem}[thm]{Lemma}
\newtheorem{prop}[thm]{Proposition}

\theoremstyle{definition}
\newtheorem{defn}[thm]{Definition}%
\theoremstyle{remark}
\newtheorem{rem}[thm]{Remark} 
\theoremstyle{plain}

\numberwithin{equation}{section}
\newenvironment{tenumerate}[1][]
  {\enumerate[label=\textup(\alph*\textup),ref=(\alph*),#1]}
  {\endenumerate}
\renewcommand{\glossarysection}[2][]{} %
\DeclarePairedDelimiter{\abs}{\lvert}{\rvert}
\DeclarePairedDelimiter{\norm}{\lVert}{\rVert}
\DeclarePairedDelimiter{\bra}{(}{)}
\DeclarePairedDelimiter{\pra}{[}{]}
\DeclarePairedDelimiter{\set}{\{}{\}}
\DeclarePairedDelimiter{\skp}{\langle}{\rangle}

\DeclareMathAlphabet{\mathup}{OT1}{\familydefault}{m}{n}
\newcommand{\dx}[1]{\mathop{}\!\mathup{d} #1}
\newcommand{\pderiv}[3][]{\frac{\mathop{}\!\mathup{d}^{#1} #2}{\mathop{}\!\mathup{d} #3^{#1}}}

\newcommand{\cA}{\ensuremath{\mathcal A}}

\newcommand{\cD}{\ensuremath{\mathcal D}}
\newcommand{\cE}{\ensuremath{\mathcal E}}

\newcommand{\cG}{\ensuremath{\mathcal G}}

\newcommand{\cL}{\ensuremath{\mathcal L}}
\newcommand{\cM}{\ensuremath{\mathcal M}}

\newcommand{\cP}{\ensuremath{\mathcal P}}
\newcommand{\GCE}{\ensuremath{\mathrm {CE}}}

\newcommand{\cW}{\ensuremath{\mathcal W}}

\newcommand{\Leb}{\ensuremath{{L}}}

\DeclareMathOperator*{\argmin}{argmin}

\newcommand{\N}{{\mathbb N}}
\newcommand{\R}{{\mathbb R}}

\newcommand{\eps}{{\varepsilon}}
\newcommand{\ft}{{f_t}}

\newcommand{\Rdiag}{\R^2_{\!\scriptscriptstyle\diagup}}
\newcommand{\cm}{\ensuremath{\mathrm{cm}}}
\newcommand{\loc}{\ensuremath{\mathrm{loc}}}
\newcommand{\Lip}{\ensuremath{\mathrm{Lip}}}

\DeclareMathOperator{\supp}{supp}

\DeclareMathOperator{\AC}{AC}

\newcommand{\ds}{\displaystyle}

\renewcommand{\tilde}{\widetilde}
\renewcommand{\bar}{\overline}
\renewcommand{\eps}{\varepsilon}

\def\XXint#1#2#3{{\setbox0=\hbox{$#1{#2#3}{\int}$}
     \vcentre{\hbox{$#2#3$}}\kern-.5\wd0}}

\title[On a novel gradient flow structure for the aggregation equation]{On a novel gradient flow structure for the aggregation equation}
\author{A. Esposito, R. S. Gvalani, A. Schlichting, and M. Schmidtchen}
\address{Antonio Esposito -- Mathematical Institute, University of Oxford, Woodstock Road, Oxford, OX2 6GG, United Kingdom.}
\address{Rishabh S. Gvalani -- Max-Planck-Institut für Mathematik in den Naturwissenschaften, Inselstraße 22, 04103 Leipzig, Germany.}
\address{Andr\'e Schlichting -- Institute for Analysis and Numerics, University of Münster, Orléans-Ring 10, 48149 Münster, Germany}
\address{Markus Schmidtchen -- Institute For Scientific Computing, Technische Universit\"at Dresden, Zellescher Weg 12-14, 01069 Dresden, Germany.}
\email{antonio.esposito@maths.ox.ac.uk}
\email{gvalani@mis.mpg.de}
\email{a.schlichting@uni-muenster.de}
\email{markus.schmidtchen@tu-dresden.de}

\begin{document}
\begin{abstract}
    The aggregation equation arises naturally in kinetic theory in the study of granular media, and its interpretation as a 2-Wasserstein gradient flow for the nonlocal interaction energy is well-known. Starting from the spatially homogeneous inelastic Boltzmann equation, a formal Taylor expansion reveals a link between this equation and the aggregation equation with an appropriately chosen interaction potential. Inspired by this formal link and the fact that the associated aggregation equation also dissipates the kinetic energy, we present a novel way of interpreting the aggregation equation as a gradient flow, in the sense of curves of maximal slope, of the kinetic energy, rather than the usual interaction energy, with respect to an appropriately constructed transportation metric on the space of probability measures. 
\end{abstract}



\maketitle

\section{Introduction}

In this work we propose a novel, rigorous interpretation of the one-dimensional aggregation equation
\begin{equation}\label{eq:initial_eq_intro}
 \partial_t f_t = \partial_v \bra*{\ft\, \partial_v W* \ft},\quad  W(v):=c \abs{v}^3,
\end{equation}
  where the probability measure $f_t$ describes the distribution of velocities of the system at time $t>0$. Here,  $c>0$ is some constant to be specified later. Equation~\eqref{eq:initial_eq_intro} has been considered in \cite{BCP97,BenedettoCagliotiCarrilloPulvirenti1998,CMcCV03,LiToscani2004} as a kinetic model for the evolution of a granular medium undergoing inelastic collisions. As we shall see in \cref{subsec:derivation-aggr-eq}, such an equation can, indeed, be formally derived from the inelastic and spatially homogeneous Boltzmann equation.

More recently, equation~\eqref{eq:initial_eq_intro} has been studied as a nonlocal interaction equation with an attractive interaction kernel in, e.g.,~\cite{CDFLS11,BLR11} and references therein, which can be obtained as the mean-field limit of a set of interacting particles,~\cite{CCH14}, or as a zero inertia limit~\cite{FetSun_zeroinlimit2015}. In this context, the interaction between individuals is described in terms of their relative positions rather than their relative velocities (i.e., relabelling `$v$' by `$x$' in~\eqref{eq:initial_eq_intro}). Moreover, it is well-known that the nonlocal interaction equation can be viewed as a $2$-Wasserstein gradient flow of the nonlocal interaction energy,~\cite{AGS08}.

This paper focuses on the kinetic description provided in \cite{BCP97}. We show that~\eqref{eq:initial_eq_intro} is a gradient flow of the kinetic energy with respect to a metric that can be understood as a generalisation of the $2$-Wasserstein distance, inspired by the approach in~\cite{DolbeaultNazaretSavare2009,ErbarBoltz} and motivated by the formal link with the inelastic Boltzmann equation.

In recent years, gradient flow structures have been proposed for several kinetic equations: for the homogeneous (elastic) Boltzmann equation~\cite{ErbarBoltz}, the linear Boltzmann equation~\cite{BasileBenedettoBertini2020}, and the homogeneous Landau equation~\cite{carrillo2020landau,AnYing2021}. See also~\cite{AguehCarlier2016} for a different gradient flow description of the inhomogeneous granular medium equation. Recently, the authors of~\cite{carrillo2021boltzmannlandau} made a connection between the gradient flow structures of the (homogeneous) Boltzmann and Landau equations. These results indicate that an appropriate gradient flow structure can link the inelastic Boltzmann equation and the aggregation equation.

In the remainder of the introduction, we give a formal sketch of the main ideas and the intuition behind our approach, with the inelastic spatially homogeneous Boltzmann equation acting as the starting point of our discussions. We commence by introducing some necessary notation and  other preliminary notions in Section~\ref{sec:prelim}. Then, in Section~\ref{sec:derivboltzmann}, we discuss the inelastic homogeneous Boltzmann equation. Moreover, we propose a formal gradient flow structure for this equation with the kinetic energy  as the natural energy functional. This is important in order to draw the connection with the aggregation equation \eqref{eq:initial_eq_intro}, via a formal Taylor expansion which we describe in Section~\ref{subsec:derivation-aggr-eq}. As a consequence, we can obtain the gradient flow structure of equation \eqref{eq:initial_eq_intro} in Section~\ref{ss:GF:structure:Agg}. We conclude the introduction in Section~\ref{ss:outline:results} with a discussion of the main results and an outline of the rest of the manuscript.

\subsection{Notation and preliminaries}\label{sec:prelim}
We use the notation $\Rdiag$ to denote the set $ \set{\bra{v,v_*}\in \R^2: v \neq v_*}$. This set often acts as our state space since it is impossible for particles to collide if they move at the same velocity and in the same direction. Furthermore, we denote by $L^p(\Omega,\mu)$, $p \geq 1$, the Lebesgue spaces on some measure space $(\Omega,\mu)$ and by $L^p(\Omega)$, $p \geq 1$, the standard Lebesgue spaces when $\Omega$ is a smooth Euclidean subdomain\footnote{In all of our applications $\Omega \in \set{\R, \Rdiag, [0,T]\times \Rdiag}$, and so all Borel measures are Radon measures.} and $\mu$ is the Lebesgue measure. In the same setting, we denote by $C^k(\Omega)$ the space of $k$-times continuously differentiable real-valued functions on $\Omega$ and $C_c^k(\Omega)$ (resp. $C_0^k(\Omega)$, $C_b^k(\Omega)$) the subspace of $C^k(\Omega)$ functions that are compactly supported (resp. vanishing at infinity, with bounded derivatives up to order $k$). 

We denote by $\cP(\Omega)$ the set of Borel probability measures on $\Omega$, and we write $\cM(\Omega)$ (resp. $\cM^+(\Omega)$) to denote finite (resp. non-negative) Radon measures on $\Omega$, where $\Omega$ is some Euclidean subdomain. Besides, for $p\geq 1$, we denote by 
\begin{equation}\label{eq:def:P1}
\begin{split}
    \cP_p(\Omega) &= \set*{ f \in \cP(\Omega) : m_p(f):=\!\int_\Omega |v|^p\dx f(v) < \infty}, \\
    \cP_p^{\cm}(\Omega) &= \set*{ f\in \cP_p(\Omega):  \!\int_\Omega v \dx f (v) = 0  } .
\end{split}    
\end{equation} 
Additionally, will denote by $d_p$, $p\geq 1$, the $p$-Wasserstein distance, \cite{V2}.
For two sequences, $\set{f_n}_n \subset \cP(\Omega)$ and $\set{U_n}_n \subset \cM(\Omega)$ as well as two elements $f \in \cP(\Omega)$ and $U\in\cM(\Omega)$, we write $f_n \to f \in \cP(\Omega)$ if, by duality with continuous and bounded functions, $g\in C_b(\Omega)$, there holds 
\begin{align}
  \int_{\Omega} g \dx f_n \to \int_{\Omega} g \dx f ,
\end{align}
as $n\to \infty$. In this case, we say $\set{f_n}_n$ converges \emph{narrowly} or \emph{weakly} to $f$. Moreover, we write $U_n \to U$ in $\cM(\Omega)$ if, by duality with continuous functions that vanish at infinity, $g\in C_0(\Omega)$, there holds
\begin{align}
  \int_{\Omega} g \dx U_n \to \int_{\Omega} g \dx U ,
\end{align}
as $n\to \infty$. When satisfied, we say $\set{U_n}_n$ converges \emph{weakly-$^\ast$} to $U$.

Likewise, we write $U_n \to^c U$ in $\cM(\Omega)$ if, by duality with continuous functions with compact support, $g\in C_c(\Omega)$, there holds
\begin{align}
  \int_{\Omega} g \dx U_n  \to \int_{\Omega} g \dx U ,
\end{align}
as $n\to \infty$. In this case, the induced topology is the \emph{vague topology}.

\subsection{The inelastic Boltzmann equation \& decay of the kinetic energy}\label{sec:derivboltzmann}

We consider the time evolution of the velocity distribution, $f_t $, of a system of particles that undergo inelastic collisions with coefficient of restitution $e \in [0,1)$. Throughout this paper, we shall denote by  $v,v_*$,  the pre-collisional velocities and by $v',v'_*$, the post-collisional velocities, respectively, which  can be computed using the following two laws: the reduction of the relative velocity of the particles due to the inelastic collisions
\begin{align}
    v'-v_*' &= -e(v-v_*),
\end{align}
and the conservation of momentum, i.e.,
\begin{align}
    v' + v'_* &= v+ v_*.
\end{align}
 The limit $e\to1$ corresponds to elastic collisions, while $e=0$ models  sticky collisions. Solving for the post-collisional velocities, $v',v'_*$, we obtain
\begin{align}
\left\{
\begin{array}{rl}
v'   \!\!\! &= \dfrac{1- e}{2}v + \dfrac{1+ e}{2}v_*, \\[1em]
v'_* \!\!\! &= \dfrac{1+ e}{2}v + \dfrac{1- e}{2}v_*.
\end{array}
\right.
\end{align}

We now define the weak form of the Boltzmann equation. We refer to the appendix for a formal derivation of the equation from a simple gain-loss argument.
\begin{defn}[Nonlocal gradient and weak form for the inelastic Boltzmann equation]
  We define the nonlocal gradient of a function $\varphi\in C^0(\R)$ as follows
  \begin{align}\label{e:def:nabla:bar}
    \bar{\nabla} \varphi(v,v_*) := \frac{\varphi(v')+ \varphi(v'_*) -\varphi(v) -\varphi(v_*)}{\abs{v-v_*}^2}(v-v_*),
  \end{align}
  for $(v, v_*)\in\Rdiag$, i.e., $\bar{\nabla} \varphi(v,v_*) : \Rdiag \to \R$. A curve $f : [0,T] \to \cP(\R)$ is a weak solution of the inelastic Boltzmann equation with collision kernel $\sigma=\sigma(|v|)$ provided that for all $\varphi \in C_c^\infty(\R)$ and almost all $t\in [0,T]$, it holds
\begin{align}
  \label{eq:WeakFormBoltzmann}
  \skp*{\varphi, \partial_t \ft } = \frac{1}{2}\iint_{\Rdiag} \sigma( \abs{v-v_*})  \bra*{v-v_*}\bar{\nabla} \varphi \dx\ft(v) \dx\ft(v_*).
\end{align}
\end{defn}
The choice of~\eqref{e:def:nabla:bar} is made such that it has the units of inverse velocity and such that it generalises to higher dimensions in a straightforward manner.
By considering its negative adjoint in the weighted space $L^2(\Rdiag,\sigma)$, we obtain a divergence acting on nonlocal fluxes $U\in \cM(\Rdiag)$ such that
\begin{align}\label{eq:Boltz:div}
  \int \varphi(v) \dx(\bar\nabla \cdot U)(v) = -\frac{1}{2} \iint \sigma(\abs{v-v_*}) \bar\nabla \varphi(v,v_*)  \dx{U(v,v_*)} . 
\end{align}
In this sense, we obtain that the weak form~\eqref{eq:WeakFormBoltzmann} can be cast into the form of a nonlocal continuity equation
$$
    \partial_t f_t + \bar \nabla \cdot U_t = 0,
$$
where the associated flux, $U_t$, is given by 
\begin{align}
    \label{eq:Boltz:flux}
  \dx U_t(v,v_*) = (v-v_*) \dx f_t(v) \dx f_t(v_*) .
\end{align}

\subsubsection*{Decay of the kinetic energy}
For a given velocity distribution, $f$, we define the kinetic energy as follows
\begin{align}
    \label{eq:KineticEnergy}
    \cE(f) := \frac12 \int_\R v^2 \dx f(v) \, .
\end{align}
Due to the fact that collisions between particles are inelastic, one would expect that the post-collisional kinetic energy is less than the pre-collisional energy. In fact, one can see that the post-collisional kinetic energy is related to the pre-collisional kinetic energy  via
\begin{equation}
    \abs{v'}^2 +\abs{v'_*}^2 = \frac{1+ e^2}{2}(\abs{v}^2 + \abs{v_*}^2) + (1-e^2)v v_*  ,
    \label{eq:el}
\end{equation}
for $e \in [0,1)$. We now use the weak formulation,~\eqref{eq:WeakFormBoltzmann}, to show that the kinetic energy decays along a solution of the
inelastic Boltzmann equation. By noting that,  $\frac{\delta\cE}{\delta f} = \frac{1}{2} v^2$, we use~\eqref{eq:el} to obtain
\begin{align}\label{eq:barnablav2}
  (v-v_*)\bar\nabla\bra*{ \abs{v}^2}&=   \abs{v'}^2 +\abs{v'_*}^2 - \abs{v}^2 - \abs{v_*}^2 = - \frac{1-e^2}{2} \abs{v-v_*}^2,
\end{align}
which, upon substituting $\varphi =\frac{\delta\cE}{\delta f} $ into~\eqref{eq:WeakFormBoltzmann}, yields
\begin{align}
    \frac{\dx{}}{\dx{t}}\cE(\ft) 
    &= -\frac{1-e^2}{8}\iint_{\Rdiag}  \sigma( \abs{v-v_*})\abs{v-v_*}^2  \dx (\ft \otimes \ft)(v, v_*) \leq 0 \, .
\end{align}
For the specific case of Maxwell molecules, that is $\sigma(|a|)=|a|$, we have
\begin{align}\label{e:Boltz:dissipation}
    \frac{\dx{}}{\dx{t}}\cE(\ft) &= -\frac{1-e^2}{8}\iint_{\Rdiag}  \abs{v-v_*}^3  \dx(\ft \otimes \ft)(v,v_*).
\end{align}
\begin{rem}
From~\eqref{e:Boltz:dissipation}, we can heuristically obtain Haff's law by considering the evolution of a family local equilibria $m_\eta$ such that $\cE(m_\eta)= \eta$. One then obtains an equation for $\eta$ of the form
\begin{align}
 \pderiv{}{t}\eta(t) = \pderiv{}{t} \cE(m_{\eta(t)})\lesssim -\eta(t)^{\frac{3}{2}},
\end{align}
which leads to
\begin{align}
  \eta(t) \lesssim \frac{1}{1+t^2}. 
\end{align}
Hence, the solutions converge on an algebraic time scale to a Dirac measure. A rigorous proof of this convergence can be found in~\cite{MischlerMouhotRicard2006,MischlerMouhot2006}. From the decay of the kinetic energy in~\eqref{e:Boltz:dissipation}, it becomes, indeed, clear that the system loses kinetic energy in the long run, i.e., it cools down. This leads to the formation of a Dirac measure as time goes to infinity, which is at the same time a minimiser of~\eqref{eq:KineticEnergy} in the space of probability measures with a fixed centre of mass. Hence, the only stationary states of the system are Dirac measures.
\end{rem}

\subsubsection*{Identification of a novel gradient structure}
\label{sec:IdentificationOOBoltzmann}

From our analysis we know that the system is driven by the kinetic energy~\eqref{eq:KineticEnergy}, whose first variation $\frac{\delta\cE}{\delta f} = \frac{v^2}{2}$ can be identified in the flux~\eqref{eq:Boltz:flux} by re-expressing it as
\begin{align}\label{eq:Boltz:flux:energy:relation}
  \dx U_t(v,v_*) =- \frac{4}{1-e^2}  \bar\nabla \frac{\delta\cE}{\delta f}(v,v_*)  \dx \bra*{ f_t \otimes f_t}(v,v_*) \, .
\end{align}
In this way, we can reformulate the homogeneous inelastic Boltzmann equation in its weak form~\eqref{eq:WeakFormBoltzmann} as
\begin{align}
  \skp{\varphi, \partial_t f_t } =- \frac{2}{1-e^2} \iint_{\Rdiag} \sigma(|v-v_*|) \bar\nabla \frac{\delta\cE}{\delta f}(v,v_*) \bar\nabla \varphi \dx (f_t \otimes f_t) (v, v_*),
\end{align}
which by the definition of the divergence from~\eqref{eq:Boltz:div} becomes
\begin{equation}
  \partial_t f_t = \frac{4}{1-e^2} \bar\nabla \cdot \bra*{ f \otimes f \ \bar\nabla \frac{\delta\cE}{\delta f}} ,
\end{equation}
whence we can identify the \emph{kinetic relation}, also called \emph{Onsager operator}, between forces\footnote{In this setting the force is understood in the generalised sense as a derivative in phase space.} and fluxes as 
\begin{equation}
  K_f \psi = -\frac{4}{1-e^2}\bar\nabla \cdot\bra*{ f \otimes f \ \bar\nabla \psi} ,
\end{equation}
which in the weak form becomes
 \begin{align}\label{e:InElastBoltz:Onsager}
  \skp{\varphi, K_f \psi} =\frac{2}{1-e^2} \iint_{\Rdiag} \sigma(\abs{v-v_*}) \bar{\nabla} \varphi \bar{\nabla} \psi \dx (f \otimes f)(v, v_*).
  \label{e:InElastBoltz:Onsager}
\end{align}
\begin{rem}[The Onsager operator for elastic Boltzmann and physical kernels]
\label{rem:OnsagerForElasticBoltzmann}
In particular, we observe that $K_f$ is only defined for $e\in [0,1)$ and becomes meaningless in the elastic limit $e\to  1$. Nevertheless, it has structural similarities to the Onsager operator introduced in~\cite{ErbarBoltz} for the homogeneous elastic Boltzmann equation.
\end{rem}

\subsection{Formal derivation of the aggregation equation}\label{subsec:derivation-aggr-eq}
This section is dedicated to a formal derivation of the aggregation equation from the inelastic Boltzmann equation. To this end, we consider the weak formulation of the inelastic Boltzmann equation, \eqref{eq:WeakFormBoltzmann}. For $v'$ close to $v_*$ and $v'_*$ close to $v$, i.e., for almost elastic collisions, i.e., $e\approx 1$, by~\eqref{e:def:nabla:bar}, we have
\begin{align}
    \label{e:nablabar:local}
    \bar{\nabla} \varphi \sim \frac{1- e}{2}\bra*{\varphi'(v_*) - \varphi'(v)} + \mathcal{O}\bra*{\abs*{\frac{1-e}{2}}^2 |v-v_*|}.
\end{align}
Substituting this into~\eqref{eq:WeakFormBoltzmann}~and disregarding all higher order terms, we obtain
\begin{align}
    \label{eq:TaylorExpandedBoltzmannInWeakForm}
    \skp{\varphi, \partial_t f_t}=  \bra*{\frac{1- e}{4}} \iint_{\Rdiag}  \sigma( \abs{v-v_*})(v-v_*)( \varphi'(v_*) - \varphi'(v)) \dx(\ft \otimes \ft)(v, v_*).
\end{align}
Letting the function $\Sigma:\R\to\R$ be such that  $\partial_v\Sigma(v)=\sigma(|v|)\, v$, the above equation simplifies to
\begin{align}
    \skp{\varphi, \partial_t f_t}=  \bra*{\frac{1- e}{4}} \iint_{\Rdiag} \partial_v \Sigma(v-v_*) (\varphi'(v_*) - \varphi'(v)) \dx(\ft \otimes \ft)(v, v_*).
    \label{eq:taylorib}
\end{align}
Unsymmetrising in $v$ and $v_*$ yields
\begin{align}
    \label{eq:TowardsAggregation}
    \begin{split}
    \skp{\varphi, \partial_t f_t}
    &= - \bra*{\frac{1- e}{2}} \iint_{\Rdiag} \partial_v \Sigma(v-v_*) \varphi'(v) \dx(\ft \otimes \ft)(v, v_*)\\
    &= - \int_{\R} \varphi'(v)  \int_{\R}\bra*{\frac{1- e}{2}} \partial_v \Sigma(v-v_*) \dx(\ft\otimes \ft)(v, v_*).
    \end{split}
\end{align}
Choosing 
$$
    W(v)=\frac{1-e}{2} \Sigma(v),
$$
it is immediate to see that  \eqref{eq:TowardsAggregation} is the weak formulation of the aggregation equation
\begin{align}
    \label{eq:AggregationEquation}
    \partial_t f_t = \partial_v \bra*{\ft\, \partial_v W* \ft}.
\end{align}
Note that for the physical kernel, $\sigma(\abs{v})=\abs{v}$, the interaction potential for the aggregation equation becomes
\begin{align}\label{e:def:W}
    W: \R &\to \R_+,\quad 
    v \mapsto \frac{1-e}{6} \abs{v}^3 \, .
\end{align}
Furthermore, we stress that this expansion relies on the fact $e<1$ as otherwise the evolution is trivial, i.e., $\skp{\varphi,\partial_t f_t}=0$ in~\eqref{eq:TaylorExpandedBoltzmannInWeakForm}.

\subsection{Formal gradient flow structure of the aggregation equation}\label{ss:GF:structure:Agg}
As previously mentioned, the aggregation equation can be cast into a 2-Wasserstein gradient flow framework (cf. e.g.~\cite{AGS08, CDFLS11}) for the nonlocal interaction energy
\begin{align}
    \cW(f) = \frac12 \iint_{\Rdiag} W(v-v_*) \dx (f\otimes f)(v, v_*) \, ,
\end{align}
which is dissipated along the flow,   \eqref{eq:AggregationEquation}, in such a way that
\begin{align}
    \frac{\dx}{\dx{t}}\cW(\ft) = - \int_\R \abs*{\partial_v W * \ft}^2 \dx\ft(v).
\end{align}
As demonstrated above, the aggregation equation can be formally derived from the inelastic Boltzmann equation. It is therefore not unreasonable to expect that  the aggregation equation is also a gradient flow for the kinetic energy defined in~\eqref{eq:KineticEnergy}. 
To this end, we study its dissipation along the flow of equation~\eqref{eq:taylorib}. For convenience, we introduce the notation
\begin{align}
    \label{eq:sigma_e}
    \sigma_e(|v-v_*|) := \frac{1-e}{4} |v-v_*|,
\end{align}
which we shall use throughout this work. Setting $\varphi(v)=v^2/2$ we have
\begin{align}
    \label{eq:KineticEnergyDissipation}
    \frac{\dx}{\dx{t}}\cE(\ft)= - \iint_{\Rdiag}  \abs{v-v_*}^2 \sigma_e(v-v_*) \dx(\ft\otimes\ft)(v, v_*)
    =: - \cD(f_t) \leq 0 \, ,
\end{align} 
where $\cD: \cP(\R) \to [0,+\infty]$ is the so-called \emph{dissipation functional}. Thus, the kinetic energy is a Lyapunov function for the dynamics of the aggregation equation. 

The preceding computation reveals an energy-dissipation structure of the aggregation equation with respect to the kinetic energy, cf.  \eqref{eq:KineticEnergyDissipation}, which suggests there may exist an appropriate metric for which~\eqref{eq:AggregationEquation} is a gradient flow of $\cE(f)$. Next, we identify the Onsager operator for this metric and, using the new formalism, derive the weak form of the aggregation equation. %
More precisely, \eqref{eq:TaylorExpandedBoltzmannInWeakForm} becomes
\begin{align}
  - \skp{ \varphi, \partial_t f_t}
  &=-\iint_{\Rdiag}  \sigma_e(\abs{v-v_*})(v-v_*)\bra*{\varphi'(v_*) - \varphi'(v)} \dx(\ft\otimes \ft)(v_*)\notag\\
  &=\iint_{\Rdiag} \sigma_e(\abs{v-v_*})\bra*{\partial_v \frac{\delta \cE}{\delta f}(v_*) - \partial_v \frac{\delta \cE}{\delta f}(v)} \bra*{\varphi'(v_*) - \varphi'(v)} \dx(\ft\otimes \ft)(v, v_*) \label{eq:weakAgg}\\
  &=\skp*{\varphi, K^{\rm agg}_{\ft} \mathrm{D}\cE}\notag,
\end{align}
where $\sigma_e$ is as in  \eqref{eq:sigma_e}. Then, we can read off the appropriate Onsager operator in its weak form
\begin{align}
  \label{e:Aggregate:Onsager}
  \skp{\varphi, K^{\rm agg}_{f} \psi}
  \!=\! \iint_{\Rdiag} \sigma_e(|v-v_*|) \bra*{\varphi'(v_*) - \varphi'(v)} \, \bra*{\psi'(v_*) - \psi'(v)}\dx(\ft \otimes \ft)(v, v_*) \, ,
\end{align}
where $\sigma_e$ is as in  \eqref{eq:sigma_e}, $f\in \cP(\R)$, $\varphi \in C_c^1(\R)$ a test function, and $\psi\in C_c^1(\R)$  a driving vector field.\\

By virtue of~\eqref{e:Aggregate:Onsager}, we note that the Onsager operator induces a positive-definite ($\skp{\varphi, K_f^{\rm agg} \varphi} \geq 0$), bilinear form which is structurally similar to the operator in to  ~\eqref{e:InElastBoltz:Onsager}. The similarity with the Onsager operator of Section \ref{sec:IdentificationOOBoltzmann} is in particular seen since, up to a multiplicative constant, one can be obtained from the other by replacing $\tilde{\nabla}$ by $\bar{\nabla}$ or vice-versa, cf.  ~\eqref{e:nablabar:local}, i.e., $\bar\nabla \varphi \approx \frac{1-e}{2} \tilde\nabla \varphi$. This observation gives rise to the following definition.
\begin{defn}[Nonlocal-local gradient]\label{def:nl_grad}
  For any function $\varphi \in C^1(\R)$ we define its \emph{nonlocal-local gradient} $\tilde{\nabla} \varphi : \Rdiag \to \R$ by
  \begin{equation} 
    \label{eq:nl_grad}
    \tilde{\nabla}\varphi(v,v_*)
    =\varphi'(v_*)-\varphi'(v), \qquad \text{for all } (v,v_*)\in \Rdiag . 
  \end{equation}
\end{defn}
Using this definition, we revisit  \eqref{e:Aggregate:Onsager}, which now reads 
\begin{align}\label{eq:Onsager:tildenabla}
  \skp{\varphi, K^{\mathrm{agg}}_{f} \psi}
  = \iint_{\Rdiag} \sigma_e(|v-v_*|) \tilde \nabla \varphi(v,v_*) \, \tilde \nabla \psi(v,v_*)  \dx (f \otimes f)(v, v_*).
\end{align}
Based on this definition, \eqref{eq:KineticEnergyDissipation} can be written as
\begin{align}\label{eq:energy-dissipation:intro}
    \frac{\dx}{\dx{t}}\cE(\ft)= - \iint_{\Rdiag}  \abs*{\tilde \nabla \frac{\delta \cE}{\delta f}}^2 \sigma_e(v-v_*) \dx(\ft\otimes \ft)(v, v_*)
    =: - \cD(f_t) \leq 0.
\end{align} 
\begin{rem}[Connection to graphs]
    From Definition~\ref{def:nl_grad}, we can read a continuous graph structure $(\R,\Rdiag)$, where $\R$ is the set of vertices and $\Rdiag$ that of edges, equipped with an operator $\tilde\nabla:C^1(\R)\to C(\Rdiag)$ connecting test functions on vertices to test functions on edges. This gives rise to the negative dual operator, which we interpret as a divergence $\tilde\nabla\cdot: \cM(\Rdiag)\to \cM(\R)$ connecting a flux on the edge set $\Rdiag$ to an infinitesimal change of state, i.e., a tangential direction (see Definition~\ref{def:nl_div}). Moreover, note that the driving force field $\tilde\nabla\delta_f \cE(v,v_*)$ is in our case not just a difference of potential values at $\delta_f\cE(v)$ and $\delta_f\cE(v_*)$, as it is the case for simple graph gradients (see e.g.~\cite{EspPatSchSle}), but rather a difference of rates $(\delta_f\cE)'$. It is in this sense that $\tilde\nabla$ is nonlocal-local.
\end{rem}

\subsection{Outline and results}\label{ss:outline:results}

In this paper, we show that the kinetic energy~\eqref{eq:KineticEnergy} is not merely a Lyapunov functional for the aggregation equation as was shown in~\eqref{eq:energy-dissipation:intro}. Indeed, the aggregation equation can be cast into a rigorous metric gradient flow setting where a dynamical transport cost induces the metric in the spirit of  \cite{BenamouBrenier2000,DolbeaultNazaretSavare2009}, and the kinetic energy acts as the driving energy functional.

The variational description we propose provides a promising setting to make rigorous the link with the inelastic spatially homogeneous Boltzmann equation, i.e., to rigorously derive the aggregation of particles from the inelastic spatially homogeneous Boltzmann equation, as was formally shown in~\cite{BCP97}. 
This investigation is kept for future work, along with an extension of our results to more general and singular collision kernels, as well as to higher dimensions, following, e.g., \cite{Toscani04_d3}.

We start by introducing a generalised notion of the continuity equation based on the aforementioned nonlocal-local operators, $\tilde \nabla$, and its formal negative adjoint, $\tilde \nabla \cdot$ (cf. Definition~\ref{def:nl_div}). This consists of a pair $\set*{(f_t, U_t)}_{t\in [0,T]}\subset \cP(\R)\times \cM(\Rdiag)$ satisfying, in a suitable measure-valued sense (see Definition~\ref{def:CRE}), the equation
\begin{equation}\label{eq:gce}
  \partial_t f_t +\tilde{\nabla}\cdot U_t=0, \qquad \mbox{on } [0,T]\times\R . \tag{CE}
\end{equation}
Using the definition of the Onsager operator in~\eqref{eq:Onsager:tildenabla}, we then introduce an action-density functional, $\cA:\cP(\R)\times \cM(\Rdiag)\to [0,\infty]$, which gives rise to a dynamical transport cost, $d_\cA(\mu_0,\mu_1)$, by minimising the total action of a curve $\set*{(f_t, U_t)}_{t\in[0,1]}$ connecting two measures $\mu_0,\mu_1\in \cP(\R)$ and satisfying~\eqref{eq:gce}, cf.~Theorem~\ref{thm:metric}.

Moreover, in this metric setting, we are able to provide a characterisation of weak solutions to the aggregation equation in the form~\eqref{eq:weakAgg} as curves of maximal slope. To this end, we define along any curve $\set*{(f_t, U_t)}_{t\in[0,T]}$  of finite action staisfying~\eqref{eq:gce} the so-called \emph{De Giorgi functional}
\begin{align}
  \cG_T(f) = \cE(f_T)- \cE(f_0) + \frac12 \int_0^T \cA(f_t, U_t) \dx t + \frac12 \int_0^T \cD(f_t) \dx t \geq 0
\end{align}
where the non-negativity is the consequence of a suitable chain rule (see Lemma~\ref{lem:crule}).
The weak solutions to~\eqref{eq:weakAgg} are found to be elements of the zero locus of the De Giorgi functional, i.e., $\cG_T(f)=0$. Conversely, any element of the zero locus of the De Giorgi functional is necessarily a weak solution to the aggregation equation (see Theorem~\ref{thm:weak-curves}).
Finally, we prove that curves of maximal slope are stable with respect to convergence of the initial measures $\mu_0^n \to \mu_0$ such that $\cE(\mu_0^n)\to \cE(\mu_0)$ (cf.~Theorem~\ref{thm:stability}). This allows us to prove the existence of gradient flow solutions based on a finite-dimensional particle approximation (see~\cref{lem:particle}).

\section{The nonlocal-local continuity equation and the collision metric}

\subsection{A nonlocal-local continuity equation}
\label{sec:GCE}
For the subsequent analysis, we study arbitrary curves, $\set{f_t}_{t \in [0,T]} \subset \cP(\R)$, in the set of probability measures induced by a driving field, $\psi_t$, connecting two probability measures $f_0,f_T \in \cP(\R)$. By~\eqref{e:Aggregate:Onsager}~and~\eqref{eq:Onsager:tildenabla}, we have
\begin{align}
    \skp{\varphi,\partial_t \ft }&= -\skp{\varphi, K^{\rm agg}_{\ft} \psi_t}\\
    &=- \iint_{\Rdiag} \ft(v) \ft(v_*)\sigma_e(|v-v_*|) \tilde{\nabla }\varphi(v,v_*) \tilde{\nabla }\psi_t(v,v_*)  \dx{v} \dx{v_*},
\end{align}
which we take as the basis for the definition of a nonlocal-local continuity equation~\eqref{eq:gce}. To this end, we first define an appropriate divergence as the formal adjoint of the nonlocal-local gradient from Definition~\ref{def:nl_grad}.
\begin{defn}[Nonlocal-local divergence] 
  \label{def:nl_div}
  For any $U\in \cM(\Rdiag)$, its \emph{nonlocal-local divergence} $\tilde{\nabla}\cdot U \in \cM(\R)$ 
  is defined as negative dual with weight $\sigma_e$ of $\tilde{\nabla}$, i.e., for all $\varphi\in C^1_c(\R)$ it holds
  \begin{align} 
    \label{eq:nl_div}
    \int_\R \varphi(v) \dx{(\tilde{\nabla} \cdot U)}(v) &= -\iint_{\Rdiag}\tilde{\nabla}\varphi(v,v_*)\sigma_e(|v-v_*|)\dx U(v,v_*)\\
    &=  \iint_{\Rdiag} \varphi'(v) \sigma_e(|v-v_*|) \bra*{ \dx U(v,v_*) - \dx U(v_*,v)} .
  \end{align}
\end{defn}
Now, we can define the nonlocal-local continuity equation.
\begin{defn}[Weak solution to~\eqref{eq:gce}]\label{def:CRE}
A pair $\{(\ft,U_t)\}_{t\in[0,T]}$ is called (weak) solution of the nonlocal-local continuity equation~\eqref{eq:gce}
on $[0,T]$ if there exist two families of measures $\{f_t\}_{t\in[0,T]}\subset \cP(\R)$ and $\{U_t\}_{t\in[0,T]}\subset \cM(\Rdiag)$  such that the map $t \mapsto f_t$ (resp. $t \mapsto U_t$) is measurable with respect to the weak-$^*$ topology on finite Radon measures and they satisfy the following integrability condition 
\begin{align}\label{eq:integrability-cond}
    \int_0^T\iint_{\Rdiag}\sigma_e(|v-v_*|)\dx |U_t|(v,v_*)\dx{t}<+\infty .
\end{align}
along with the weak form of the \emph{nonlocal-local continuity equation~\eqref{eq:gce}} for every $C_c^1((0,T)\times\R)$
\begin{align}
\label{eq:CREweakform}
    \int_0^T \int_\R \partial_t \varphi_t(v) \dx{f_t}(v) \dx{t} + \int_0^T \iint_{\Rdiag} \sigma_e(|v-v_*|)  \tilde{\nabla}\varphi_t(v,v_*)  \dx{U_t(v, v_*)}\dx{t} =0.
\end{align}
We denote by $\set{(f_t,U_t)}_{t\in [0,T]} \in \GCE_T(\mu_0)$ a solution of the nonlocal-local continuity equation on $[0,T]$ starting at $\mu_0$, and we write $\set{(f_t,U_t)}_{t\in [0,T]} \in \GCE_T(\mu_0,\mu_T)$ for solutions connecting $\mu_0$ with $\mu_T$. We will drop the subscript $T$ whenever $T=1$.
\end{defn}
Note that the second term in the weak formulation~\eqref{eq:CREweakform} of the~\eqref{eq:gce} is well-defined under the integrability condition~\eqref{eq:integrability-cond}, since $|\tilde{\nabla}\varphi_t(v,v_*)|\le2\|\partial_v \varphi_t(\cdot)\|_{C^0(\R)}$, for all $t\in [0,T]$. 

\begin{rem}[Strong form of~\eqref{eq:gce}]
Note that, for $U_t\ll f_t\otimes f_t$ and $f_t\ll \dx v$ for any $t\in[0,T]$,
after an integration by parts in $v$ of~\eqref{eq:CREweakform}, we arrive at
\begin{align}
    \label{eq:WeakFormGeneralisedContinuityEquation}
    \skp{\varphi,\partial_t \ft }
    &=- \int_\R \varphi(v)2\, \partial_v \bra*{ \int_\R \ft(v) \ft(v_*)\sigma_e(|v-v_*|) \tilde{\nabla }\psi_t(v,v_*)   \dx{v_*}} \dx{v} .
\end{align}
From~\eqref{eq:WeakFormGeneralisedContinuityEquation}, we have that a couple, $(\ft, \psi_t)$, consisting of the curve, $\ft$, and the driving field, $\psi_t$,  satisfies the strong form of the \emph{nonlocal-local continuity equation}  provided that
\begin{align}
  \partial_t \ft  +  2\partial_v \int_\R \ft(v)\ft(v_*)\sigma_e\bra*{\abs{v-v_*}} \tilde \nabla \psi_t \dx{v_*} = 0, 
\end{align}
where $\tilde \nabla \psi = \psi'(v_*) - \psi'(v)$, as in Definition \ref{def:nl_grad}. In the following, we will always use the weak formulation in the sense of Definition~\ref{def:CRE}.
\end{rem}

As a matter of fact, the integrability condition,  \eqref{eq:integrability-cond}, allows us to infer additional time regularity in that we can prove the existence of a continuous representative for weak solutions to the nonlocal-local continuity equation as stated in the following proposition.

\begin{prop}[Continuous representative]
Let $\{(f_t,U_t)\}_{t\in[0,T]}$ be a solution to the~\eqref{eq:gce} in the sense of Definition~\ref{def:CRE}. Then, there exists a narrowly continuous curve $[0,T]\ni t\mapsto \tilde{f}_t\in\cP(\R)$ such that $f_t=\tilde{f}_t$ for $\mathcal{L}^1$-a.e. $t\in(0,T)$ and, for any test function $\varphi \in C_c^1(\R)$, there holds
\begin{align}\label{eq:GCE:weak_form}
    \frac{\dx}{\dx t} \int \varphi(v) \dx \tilde f_t(v) = \iint_{\Rdiag} \tilde\nabla \varphi(v,v_*) \sigma_e(|v-v_*|) \dx U_t(v, v_*).
\end{align}
\end{prop}
\begin{proof}
Let $\{(f_t,U_t)\}_{t\in[0,T]}$ be a solution in the sense of Definition \ref{def:CRE}  and $\varphi \in C_c^1((0,T)\times \R)$ be a test function. Following the argument of  \cite[Lemma 8.1.2]{AGS08} or \cite[Lemma 3.1]{Erb14} by setting $V(t):=\iint_{\Rdiag}\sigma_e(|v-v_*|) \dx{| U_t|}(v,v_*)$,
we arrive at 
\begin{align}
    \label{eq:continuous-repr}
    \begin{split}
    \int_\R&\varphi_{t_2}(v)\dx{\tilde{f}_{t_2}(v)}-\int_\R\varphi_{t_1}(v)\dx{\tilde{f}_{t_1}(v)}\\
    &=\int_{t_1}^{t_2} \int_\R \partial_t \varphi_t(v) \dx{f_t}(v) \dx{t} + \int_{t_1}^{t_2} \iint_{\Rdiag} \tilde{\nabla}\varphi_t(v,v_*)\sigma_e(|v-v_*|) \dx{U_t}(v,v_*)\dx{t},
    \end{split}
\end{align}
for any $0\leq t_1 < t_2 \leq T$.
In order to obtain the expression claimed in the statement of the proposition, let us choose a sequence of test functions that are in product form and whose time-component is an approximation of the indicator on an interval $(t_1,t_2)$ with $0<t_1<t_2<T$, i.e., 
\begin{align}
    \varphi^\eps(t,v) = \psi^\eps(t) \phi(v),
\end{align}
where $\supp \psi^\eps = [t_1-\eps,t_2+\eps]$ such that $\psi^\eps(t)= 1$ for $t\in [t_1,t_2]$ and $\psi^\eps\in C_c^1([0,T]), \phi \in C_c^1(\R)$. We may, for instance, choose the following approximating sequence
\begin{align}\label{eq:timecutoff}
    \psi^\eps(t) =
    \left\{
    \begin{array}{ll}
         0, & t\in(-\infty,t_1- \eps), \\
         \eps^{-1}(t-t_1+\eps), & t\in(t_1 - \eps, t_1),\\
         1, & t\in(t_1, t_2), \\
         \eps^{-1}(t_2 + \eps - t), & t\in(t_2, t_2 + \eps), \\
         0, & t\in(t_2 + \eps, \infty).
    \end{array}
    \right.
\end{align}
Upon substituting $\varphi_\eps(t,x)$ into  ~\eqref{eq:continuous-repr}, we obtain
\begin{align}
    \begin{split}
    \int_{t_1-\varepsilon}^{t_2+\varepsilon} \int_\R \partial_t \varphi^\varepsilon(t,v) \dx{f_t}(v) \dx{t} + \int_{t_1-\varepsilon}^{t_2+\varepsilon} \iint_{\Rdiag} \tilde{\nabla}\varphi^\varepsilon(t,v,v_*)\sigma_e(|v-v_*|) \dx{U_t}(v,v_*)\dx{t}=0,
    \end{split}
\end{align}
whence
\begin{align}
    \MoveEqLeft{\abs*{\frac1\eps \bra*{\int_{t_1-\eps}^{t_1}\iint_{\Rdiag} \phi(v) \dx{f_t^n}(v)\dx{t} - \int_{t_2}^{t_2+\eps}\iint_{\Rdiag} \phi(v) \dx{f_t^n}(v)\dx{t}}}} \\
    &\leq \int_{t_1 - \eps}^{t_2+ \eps} \iint_{\Rdiag} \abs*{\tilde{\nabla} \phi(v,v_*)} \sigma_e(|v-v_*|) \dx{\abs{U_t^n}}(v,v_*)\dx{t}\\
    &\leq 2\norm{\phi'}_{C^0(\R)} \int_{t_1-\eps}^{t_2+ \eps} \iint_{\Rdiag}  \sigma_e(|v-v_*|) \dx|U_t^n|(v,v_*)\dx{t},
\end{align}
where  \eqref{eq:integrability-cond} ensures that the right-hand side is $L^1$-integrable which then acts as the modulus of absolute continuity. Letting $\eps\to 0$, we have
\begin{align}
    &\abs*{\iint_{\Rdiag} \phi(v) \dx{f_{t_1}^n}(v)\dx{t} - \iint_{\Rdiag} \phi(v) \dx{f_{t_2}^n}(v)\dx{t}} \\
    &\qquad\leq 2\norm{\phi'}_{C^0(\R)} \int_{t_1}^{t_2} \iint_{\Rdiag}  \sigma_e(|v-v_*|) \dx |U_t^n|(v,v_*)\dx{t},
\end{align}
implying the narrow continuity of $\tilde f_t$.
\end{proof}
\begin{rem}[Extension of test function class]\label{rem:extension:test-functions}
  In view of~\eqref{eq:GCE:weak_form} and the integrability condition on the flux 
  we can choose $\varphi\in \Lip(\R)$ as test-function class.
  \label{rem:w1inf}
\end{rem}

We now show two peculiar properties of the solutions to the \textit{nonlocal-local continuity equation}.
\begin{prop}[Preservation of centre of mass and bounded first moments]\label{prop:GCE:centre_mass}
  Let $f_0\in \cP(\R)$ be such that $\int v \dx f_0(v)<\infty$. Then, any $\set{(f_t,U_t)}_{t\in [0,T]}\in \GCE_T(f_0)$ preserves the centre of mass, that is for all $t\in [0,T]$ it holds
  \begin{equation}\label{eq:centre-of-mass}
    \int_\R v \dx f_t(v) = \int_\R v \dx f_0(v). 
  \end{equation}
  Likewise, if $f_0\in \cP(\R)$ is such that $\int |v| \dx f_0(v)<\infty$, then any $\set{(f_t,U_t)}_{t\in [0,T]}\in \GCE_T(f_0)$ satisfies for all $t\in [0,T]$ the bound
  \begin{equation}
  \label{e:m1:GCE}
    \abs*{\pderiv{}{t} \int_\R |v| \dx f_t( v)} \leq   2 \iint_{\Rdiag}\sigma_e(|v-v_*|)\dx |U_t|(v,v_*).
  \end{equation}
\end{prop}
\begin{proof}
Let $R>0$ and let us consider the function $\varphi_R:\R \to \R$ defined as
\begin{align}\label{eq:space-cutoff}
    \varphi_R(v) =
    \left\{
    \begin{array}{ll}
         0, & v\in(-\infty, -2R), \\
         -2R-v, & v\in(-2R, -R),\\
         v, & v\in(-R, R), \\
         2R-v, & v\in(R, 2R)  \\
         0, & v\in(2R, \infty).
    \end{array}
    \right.
\end{align}
Note that
\begin{align}
    \abs*{ \tilde\nabla \varphi_R(v,v_*)} \leq 2 , \qquad\text{for almost all $(v,v_*)\in \Rdiag$};
\end{align}
while, at the same time
\begin{align}
    \abs*{ \tilde\nabla \varphi_R(v,v_*)} = 0, \qquad\text{for $ (v,v_*) \in [-R,R]^2 $.}
\end{align}
By~\cref{rem:w1inf}, this is an admissible test function in~\eqref{eq:GCE:weak_form} and  we can estimate
\begin{align}
    \biggl|\int \varphi_R(v) \dx f_t(v) &- \int \varphi_R(v) \dx f_0(v)\biggr| \\
    &= \biggl| \int_0^t \iint_{\Rdiag} \bra*{ \varphi'_R(v_*) - \varphi'_R(v)} \sigma_e(|v-v_*|) \dx U_s(v, v_*) \dx{s} \biggr| \\
    &\le 2 \int_0^t \iint_{\Rdiag \setminus [-R,R]^2} \sigma_e(|v-v_*|) \dx \abs*{U_s}(v,v_*) \dx s \to 0, \quad\text{as } R\to \infty.
\end{align}
Since $\int \varphi_R(v) \dx f_0(v) \to \int v \dx f_0(v) \in \R$, this concludes the proof of the preservation of the centre of mass. The bound for the first moment, follows from a similar construction, by choosing $|\varphi_R|$, with $\varphi_R$ as in~\eqref{eq:space-cutoff}, to be the test function in~\eqref{eq:GCE:weak_form}. Indeed, we note $\abs*{\tilde\nabla\abs*{\varphi_R}(v,v_*)} \leq 2$, for almost all $(v,v_*)\in \Rdiag$. Hence, for any $0\leq s < t\leq T$ we have
\begin{align}
	\biggl|\int \abs*{\varphi_R(v)} \dx f_t(v) - \int \abs*{\varphi_R(v)} \dx f_s(v)\biggr| \le 2 \int_s^t \iint_{\Rdiag} \sigma_e(|v-v_*|) \dx \abs*{U_s}(v,v_*) \dx s.
\end{align}
Then, we obtain the bound~\eqref{e:m1:GCE} after dividing by $t-s$, letting $t\to s$, and noting that $\abs*{\varphi_R(v)}\to \abs{v}$ as $R\to \infty$.
\end{proof}
In the following proposition, we provide a sufficient condition for the existence of a weak solution to the nonlocal-local continuity equation. In particular, any curve that is absolutely continuous with respect to $2$-Wasserstein distance, denoted by $d_2$, connecting two probability measures $\mu_0$ and $\mu_T$, and preserving the centre of mass, is also a weak solution to~\eqref{eq:gce}.
\begin{prop}[Existence of weak solutions]
Let $\mu_0,\mu_T \in \cP(\R)$ be with equal centre of mass, i.e., $\int v \dx \mu_0(v) = \int v \dx \mu_T(v)$, and $d_2(\mu_0,\mu_T)<\infty$. Then, there exists $\{(f_t, U_t)\}_{t\in[0,T]} \in \GCE_T(\mu_0, \mu_T)$.
\end{prop}
\begin{proof}
Since the $2$-Wasserstein distance is finite, i.e., $d_2(\mu_0,\mu_T)<\infty$, there exists an absolutely continuous curve $f_t: [0,T] \to \cP(\R)$ connecting $\mu_0$ and $\mu_T$ preserving the centre of mass and a vector field $V \in  \Leb^2(0,T; \Leb^2(\R, \dx{f_t}))$ such that the flux $\dx{C_t} =V_t \dx{f_t}$ satisfies for a.e.~$t\in[0,T]$
\begin{align}
    \frac{\dx}{\dx t} \int_\R \varphi(v) \dx{f_t}(v) =\int_{\R} \partial_v \varphi(v) \dx{C_t}(v) ,
\end{align}
for all $ \varphi \in C_c^1(\R)$. Note that we may simply take the 2-Wasserstein geodesic as such a curve. By a similar argument as in the proof of Proposition~\ref{prop:GCE:centre_mass} using the test-function~\eqref{eq:space-cutoff}, from the preservation of the centre of mass we obtain that $C_t$ has mean zero, that is for a.e. $t\in [0,T]$ it holds $\int_{\R} \dx C_t = 0$. The well-posedness of the weak form follows by noting that
\begin{equation}\label{eq:CE:bound}
  \int_0^T \int_\R \dx{|C_t|}(v)\dx{t} = \int_0^T \int_\R |V_t| \dx f_t(v) \dx{t} \leq T^{\frac{1}{2}} \norm{V}_{\Leb^2(0,T; L^2(\R, \dx f_t))} < \infty . 
\end{equation}
We define for all $t\in [0,T]$ the flux $U_t \in \cM(\Rdiag)$ by
\begin{equation}\label{e:def:Ct:Ut}
  \dx{U_t}(v,v_*) := \frac{1}{2\sigma_e(\abs{v-v_*})}\bra*{ \dx{f_t}(v) \dx{C_t}(v_*) -\dx{C_t}(v) \dx{f_t}(v_*)} .
\end{equation}
We can check that the resulting pair satisfies $(f_t, U_t)_{t\in [0,T]} \in \GCE_T(\mu_0,\mu_T)$. First, we check the weak form~\eqref{eq:GCE:weak_form} for which we take any $\varphi \in C_c^1(\R)$ and obtain
  \begin{align}
    \frac{\dx}{\dx{t}}\int_\R \varphi(v) \dx{\ft}(v) \dx{t} &= \iint_{\Rdiag} \tilde{\nabla}\varphi(v,v_*) \sigma_e(|v-v_*|) \dx{U_t(v, v_*)} \\
    &=  \iint_{\Rdiag}(\varphi'(v_*)-\varphi'(v))\frac{1}{2}\bra*{   \dx{C_t}(v_*) \dx{f_t}(v)- \dx{C_t}(v) \dx{f_t}(v_*)} \\
    &=  \int_{\R}\partial_v \varphi(v) \dx{C_t}(v),
  \end{align}
  where we have used the fact that $\int_\R \dx{C_t}(v)=0$.
  Second, we check the integrability condition~\eqref{eq:integrability-cond} and bound
  \begin{align}
      \int_0^T \iint_{\Rdiag} \sigma_e(|v-v_*|) |U_t(\dx v, \dx v_*)| \dx t &= \frac{1}{2} \int_0^T \iint_{\Rdiag} \bra*{ \dx f_t(v) \dx|C_t|(v_*)  + \dx|C_t|(v) \dx f_t(v_*)} \dx{t} \\
      &\leq \int_0^T \int_{\R} \dx | C_t|(v) < \infty ,
  \end{align}
  by the bound~\eqref{eq:CE:bound}.
\end{proof}

\subsection{The action-density functional and its properties}
This section is dedicated to introducing the action-density functional which plays a crucial role in the subsequent analysis.
We start by considering the auxiliary function $\alpha:\R_+ \times \R \to \R_+$ given by
\begin{equation}\label{eq:alpha-fun}
    \alpha( s, u):=
\begin{cases}
         \frac{u^2}{s},& \text{if }  s>0,  \\
         0, & \text{if }  u= 0,\\
         +\infty, & \text{if } u \neq 0, s = 0.
\end{cases}
\end{equation}
We observe that $\alpha$ is jointly convex, lower semicontinuous, and $1$-homogeneous.

Following the strategy of \cite{DolbeaultNazaretSavare2009, Erb14, ErbarBoltz, EspPatSchSle}, we define the action-density functional.
\begin{defn}[Action-density functional]\label{def:action-functional}
For any $f\in\cP(\R)$ and $U\in\cM(\Rdiag)$, set $|\lambda|=f\otimes f+|U|\in\cM^+(\Rdiag)$. We define the action-density functional by
\begin{align}
    \cA(f,U):=\iint_{\Rdiag}\alpha\biggl(\frac{\dx{f\otimes f}}{\dx{|\lambda|}},\frac{\dx{U}}{\dx{|\lambda|}}\biggr)\sigma_e(\abs{v-v_*})\dx{|\lambda|}(v,v_*) \, ,
\end{align}
where the function $\alpha$ is defined as in  \eqref{eq:alpha-fun}.
\end{defn}
Note that the above definition is independent of the choice of $|\lambda|$ as long as $f\otimes f+|U|\ll|\lambda|$. In the next lemma, we see that the flux of any couple, $(f,U)$, with finite action-density, takes a specific form. %
\begin{lem}
\label{prop:action}
Let $f \in \cP(\R)$ and $U\in\cM(\Rdiag)$ be such that $\cA(f,U)<+\infty$. Then, there exists a Borel function $\hat{U}:\Rdiag\to\R$ such that
\[
    \dx{U}(v,v_*)= \hat {U}(v,v_*)\dx{\bra{f \otimes f}}(v,v_*) \, ,
\]
and the action-density is given by
\begin{align}
\label{eq:action_final_form}
    \cA(f,U)&=\iint_{\Rdiag}|\hat{U}|^2(v,v_*) \sigma_e(\abs{v-v_*})\,\dx{\bra{f \otimes f}}(v, v_*)  \, .
\end{align}
In particular, if $f\ll \cL$ then $U\ll \cL \otimes \cL$, as well.
\end{lem}

\begin{proof}
Let $f \in \cP(\R)$, $U \in \cM(\Rdiag)$, and  $|\lambda|\in \cM^+(\Rdiag)$ be as in Definition ~\ref{def:action-functional} such that $\cA(f, U) < \infty$. Then, setting $\mu:=f\otimes f$, the action functional can be written as
\begin{align}
    \cA(f, U) 
    &= \iint_{\Rdiag} \alpha\bra*{\frac{\dx{\mu}}{\dx{\abs{\lambda}}}, \frac{\dx{U}}{\dx{\abs{\lambda}}}} \sigma_e(\abs{v-v_*}) \dx{\abs{\lambda}}
    = \iint_{\Rdiag} \alpha\bra*{\tilde{\mu}, \tilde{U}} \sigma_e(\abs{v-v_*}) \dx{\abs{\lambda}} \, ,
\end{align}
where $\tilde{\mu}, \tilde{U}$ are the Radon--Nikodym derivatives of $\mu,U$, respectively, with respect to $\abs{\lambda}$. In order to be able to use the $1$-homogeneity of the kernel, $\alpha$, we show that $U \ll \mu$. To this end, let $N\subset \Rdiag$ be a $(\sigma_e\mu)$-null set, i.e., $\tilde \mu(v, v_*) = 0$, for $v,v_*\in N$, $\sigma_e\abs{\lambda}$-a.e. in $\Rdiag$. Since the action of $(f, U)$ is  finite, we conclude, by definition of $\alpha$, cf. \eqref{eq:alpha-fun}, that $\tilde U(v, v_*) = 0$, $\sigma_e\abs{\lambda}$-a.e., which, in turn, implies $ U \ll \mu$. Upon an application of the chain rule we obtain
\begin{align}
    \frac{\dx{U}}{\dx{\abs{\lambda}}} = \frac{\dx{U}}{\dx{\mu}} \frac{\dx{\mu}}{\dx{\abs{\lambda}}} =: \hat U \tilde \mu.
\end{align}
Substituting this expression into the action density above in conjunction with the homogeneity of order one, we obtain
\begin{align}
    \cA(f,U) 
    &= \iint_{\Rdiag} \abs{\hat U}^2 \tilde \mu \,\sigma_e(|v-v_*|) \dx{\abs{\lambda}}
    = \iint_{\Rdiag} \abs{\hat U}^2\,\sigma_e(|v-v_*|) \dx{\mu} \\
    &= \iint_{\Rdiag} \abs{\hat U}^2 \sigma_e(|v-v_*|) \dx{\bra{f\otimes f}}(v,v_*),
\end{align}
which concludes the proof.
\end{proof}

\begin{prop}[Antisymmetric fluxes have lower action]\label{prop:antisymm-flux}
Let $f \in \cP(\R)$ and $U \in \cM(\Rdiag)$ be such that $\cA(f, U) < \infty$. Then, there exists an antisymmetric\footnote{That is to say $U(A)=-U(\Gamma(A))$, for all Borel $A \subset \Rdiag$, where $\Gamma(v,v_*)=(v_*,v)$.} measure $U^{\mathrm{as}} \in \cM(\Rdiag)$, $U^{\mathrm{as}} \ll \mu$, such that
\begin{align}
    \cA(f, U^{\mathrm{as}})\leq \cA(f, U),\quad \mbox{ and }\quad \tilde{\nabla} \cdot U^{\mathrm{as}}= \tilde{\nabla}\cdot U.
\end{align}
\end{prop}
\begin{proof}
We define $\hat{U}^{\mathrm{as}} : \Rdiag \to \R$ to be
\begin{align}
    \hat{U}^{\mathrm{as}}(v,v_*) := \frac{1}{2}\bra*{\hat{U}(v,v_*)- \hat{U}(v,v_*)},
\end{align}
where $\hat{U}$ is as defined in the statement of~\cref{prop:action}. 
This defines a measure, $U^{\mathrm{as}} \in \cM(\Rdiag)$, via the relation
\begin{align}
    \dx{U^{\mathrm{as}}}(v,v_*):=\hat{U}^{\mathrm{as}}(v,v_*) \dx\bra*{f \otimes f}(v,v_*).
\end{align}
The proof then follows by substitution. We have that
\begin{align}
    \cA(f, U^{\mathrm{as}}) 
    &= \iint_{\Rdiag} \abs[\big]{\hat{U}^{\mathrm{as}}}^2 (v,v_*)  \sigma_e(\abs{v-v_*}) \dx\bra*{f \otimes f}(v,v_*) \\
    & = \frac{1}{2}\iint_{\Rdiag} \abs[\big]{\hat{U}}^2(v,v_*)   \sigma_e(\abs{v-v_*}) \dx\bra*{f \otimes f}(v,v_*)\\
    & \quad - \frac{1}{2} \iint_{\Rdiag} \hat{U}(v,v_*) \hat{U}(v_*,v)   \sigma_e(\abs{v-v_*}) \dx\bra*{f \otimes f}(v,v_*).
\end{align}
Applying Young's inequality, we obtain
\begin{align}
    \cA(f, U^{\mathrm{as}}) 
    &\leq \frac{1}{2}\iint_{\Rdiag} \abs[\big]{\hat{U}}^2(v,v_*)   \sigma_e(\abs{v-v_*}) \dx\bra*{f \otimes f}(v,v_*) \\
    &\quad  + \frac{1}{4} \iint_{\Rdiag} \abs[\big]{\hat{U}}^2(v,v_*)   \sigma_e(\abs{v-v_*}) \dx\bra*{f \otimes f}(v,v_*) \\
    &\quad +\frac{1}{4} \iint_{\Rdiag} \abs[\big]{\hat{U}}^2(v_*,v)   \sigma_e(\abs{v-v_*}) \dx\bra*{f \otimes f}(v,v_*) \\
     &= \iint_{\Rdiag} \abs[\big]{\hat{U}}^2(v,v_*)   \sigma_e(\abs{v-v_*}) \dx\bra*{f \otimes f}(v,v_*) \\
    &= \cA(f, U).
\end{align}
Finally, we can check that, for any test function $\varphi \in C^\infty_c(\R)$, it holds that
\begin{align}
    \MoveEqLeft\iint_{\Rdiag} \tilde{\nabla} \varphi(v,v_*)\sigma_e(|v-v_*|) \dx{U^{\mathrm{as}}}(v,v_*) \\
 	&= \frac{1}{2} \iint_{\Rdiag} \tilde{\nabla} \varphi(v,v_*)\sigma_e(|v-v_*|) \dx\bra*{U(v,v_*) - U(v_*,v)} \\
    &=\frac{1}{2} \iint_{\Rdiag} \tilde{\nabla} \varphi(v,v_*)\sigma_e(|v-v_*|)\dx{U(v,v_*)}
     - \frac{1}{2} \iint_{\Rdiag} \tilde{\nabla} \varphi(v,v_*)\sigma_e(|v-v_*|) \dx{U(v_*,v)} \\
    &=\frac{1}{2} \iint_{\Rdiag} \tilde{\nabla} \varphi(v,v_*)\sigma_e(|v-v_*|) \dx{U(v,v_*)}
    +\frac{1}{2} \iint_{\Rdiag} \tilde{\nabla} \varphi(v,v_*)\sigma_e(|v-v_*|) \dx{U(v,v_*)} 
    \\
    &=\iint_{\Rdiag} \tilde{\nabla} \varphi(v,v_*)\sigma_e(|v-v_*|) \dx{U}(v,v_*),
\end{align} 
where in the penultimate step we have used the fact that $\tilde{\nabla}\varphi(v,v_*)=-\tilde{\nabla}\varphi(v_*,v)$ from Definition~\ref{def:nl_grad}. Using~\cref{def:nl_div}, the result follows. 
\end{proof}

\begin{prop}[Lower semicontinuity of the action density]
\label{prop:lsc-action} The action-density functional is lower semicontinuous with respect to the weak-$^*$ convergence in $\cP (\R)\times \cM(\Rdiag)\subset \cM(\R \times \Rdiag)$.
\end{prop}
\begin{proof}
Let us consider $\{f_n\}_{n\in\mathbb{N}}\subset \cP(\R)$ and $\{U_n\}_{n\in\mathbb{N}}\subset \cM(\Rdiag)$ such that
\begin{align}
   f_n\to f, \quad \mbox{in } \cP(\R),
\end{align}
as well as
\begin{align}
   U_n\to U, \quad \mbox{in } \cM(\Rdiag).
\end{align}
Obviously, convergence in $\cP(\R)$ of $\{f_n\}_{n\in\mathbb{N}}$ implies that $\{f_n\otimes f_n\}_{n\in\mathbb{N}}$ converges weakly-$^*$ in $\cP(\Rdiag)$. Let us define the function $g:\Rdiag\times (\R_+ \times \R)\to\R$ as
\[
    g((v,v_*), (s,u))=\alpha(s, u)\sigma_e(\abs{v-v_*}),
\]
which is lower semicontinuous in all its variables, jointly convex, and 1-positive homogeneous in $(s,u)$. Then, \cite[Theorem 3.4.3]{But89} implies the action is  weakly-$^*$ sequentially lower semicontinuous in $\cM(\R \times \Rdiag)$. 
\end{proof}

\begin{prop}[Convexity of the action density]
\label{prop:convexity-action}
Let $f^i\in\cP(\R)$ and $U^i\in\cM(\Rdiag)$ for $i=0,1$. For any $\tau\in[0,1]$, such that $f_\tau:=(1-\tau)f^0+\tau f^1$ and $U_\tau:=(1-\tau) U^0+\tau U^1$ it holds
\[
    \cA(f_\tau,U_\tau) \leq (1-\tau) \cA(f^0,U^0) + \tau \cA(f^1, U^1).
\]
\end{prop}
\begin{proof}
Let us set $\mu^i:=f^i\otimes f^i$ and consider $|\lambda|\in\cM^+(\Rdiag)$ such that $\dx{\mu^i}=\tilde{\mu}^i \dx{|\lambda|}$ and $\dx{U^i}=\tilde{U}^i\dx{|\lambda|}$, cf. Definition \ref{def:action-functional}, for instance. As consequence we have $\dx{\mu_\tau}=\tilde{\mu}_\tau \dx{|\lambda|}$ and $\dx{U_\tau}=\tilde{U}_\tau\dx{|\lambda|}$, where
\begin{align}
    &\tilde \mu_\tau:=(1-\tau)\tilde{\mu}^0+\tau \tilde{\mu}^1,\\
    &\tilde{U}_\tau:=(1-\tau)\tilde{U}^0 + \tau \tilde{U}^1.
\end{align}
The result follows by using the convexity of the function $\alpha$:
\begin{align}
    \cA(f_\tau,U_\tau)
    &=\iint_{\Rdiag}\alpha\left(\tilde \mu_\tau,\tilde U_\tau\right)\sigma_e(\abs{v-v_*})\,\dx{|\lambda|}(v,v_*)\\
    &\le(1-\tau)\iint_{\Rdiag}\alpha\left(\tilde{\mu}^0,\tilde{U}^0\right)\sigma_e(\abs{v-v_*})\,\dx{|\lambda|}(v,v_*)\\
    &\quad+\tau\iint_{\Rdiag}\alpha\left(\tilde{\mu}^1,\tilde{U}^1\right)\sigma_e(\abs{v-v_*})\,\dx{|\lambda|}(v,v_*)\\
    &=(1-\tau)\cA(f^0,U^0)+\tau\cA(f^1,U^1). \qedhere
\end{align}
\end{proof}

\subsection{Curves of finite action}
This section is dedicated to revisiting~\eqref{eq:gce} introduced in Definition~\ref{def:CRE} and presenting some of its properties.
\begin{lem}[Curves of finite action]\label{lem:finite_action:GCE}
  Let $\set{(f_t,U_t)}_{t\in [0,T]}$ be a solution to the nonlocal-local continuity equation in the sense of Definition~\ref{def:CRE} with initial datum $\mu_0 \in \cP(\R)$ not necessarily satisfying the integrability condition~\eqref{eq:integrability-cond}, but satisfying $\int_0^T \cA(f_t,U_t) \dx{t} < \infty$ and $\int_\R |v| \dx \mu_0(v)< \infty$, then $\set{(f_t,U_t)}_{t\in [0,T]}\in \GCE_T(\mu_0)$. 

  In particular, if $\mu_0\in \cP_1(\R)$, then $f_t \in \cP_1(\R)$ and the following estimate holds for all $t\in [0,T]$ 
  \begin{equation}\label{e:m1:action}
    \abs*{\pderiv{}{t} m_1(f_t)^{\frac{1}{2}}} \leq \bra*{\frac{1-e}{2}}^{\frac12} \cA(f_t,U_t)^{\frac{1}{2}}.
  \end{equation}
\end{lem}
\begin{proof}
  The proof follows by applying the bound~\eqref{e:m1:GCE} in Proposition~\ref{prop:GCE:centre_mass} for which we  further need to bound, for almost every $t\in [0,T]$, the total variation norm of the flux by a suitable Cauchy-Schwarz inequality:
  \begin{align}
     \frac{1}{2} \abs*{\pderiv{}{t} m_1(f_t)} &\leq \iint_{\Rdiag} \sigma_e(|v-v_*|) \dx |U_t|(v,v_*) \\
     &= \iint_{\Rdiag} \sigma_e(|v-v_*|)  |\hat U_t(v,v_*)| \dx(f_t\otimes f_t)(v,v_*) \\
     &\leq \cA(f_t,U_t)^{\frac{1}{2}} \bra*{\iint_{\Rdiag} \sigma_e(|v-v_*|) \dx (f_t\otimes f_t)(v,v_*) }^{\frac{1}{2}} \\
     &\leq \bra*{\frac{1-e}{4}}^{\frac12} \cA(f_t,U_t)^{\frac{1}{2}} \bra*{\iint_{\Rdiag} (|v| + |v_*|)\dx (f_t\otimes f_t)(v,v_*)}^{\frac{1}{2}} \\
     &\leq \bra*{\frac{1-e}{2}}^{\frac12} m_1(f_t)^{\frac{1}{2}} \cA(f_t,U_t)^{\frac{1}{2}} . \qedhere
  \end{align}
\end{proof}
In the next result, we associate to a given curve $(U_t)_{t\in[0,T]}$ a measure $U\in\cM\bra{\pra{0,T} \times\Rdiag}$ by setting $\dx U(t,v,v_*)=\dx U_t(v,v_*)\dx{t}$, for $(t,v,v_*)\in[0,T]\times\Rdiag$.
\begin{prop}[Compact subsets of $\GCE_T$]\label{prop:compactCRE}
Let $\set{(f_t^n,U_t^n)_{t\in [0,T]}}_{n \in \N} \subset \GCE_T(f_0^n,f_T^n)$ 
and assume there exists a constant $0<C<\infty$ such that
\begin{align}\label{ass:compact:finite:action}
    \sup_{n \in \N}\int_0^T \cA(f_t^n,U_t^n) \dx{t}  < C , \qquad \text{and} \qquad \sup_{n \in \N}  \int |v| \dx(f_0^n+ f_T^n)(v) < C \, .
\end{align}
Then, there exists $\set{(f_t,U_t)}_{t\in [0,T]} \in \GCE_T(f_0,f_T)$, and, for all $t \in [0,T]$, along a subsequence (not relabelled)
\begin{align}
    f_t^n &\to f_t, \quad\text{ in } \cP(\R),
\intertext{as well as}
    U^n &\to^c U, \quad\text{ in } \cM_\loc\bra{\pra{0,T} \times\Rdiag}.
\end{align}
Moreover, the \emph{action} is lower semicontinuous along the above subsequences $\{f^n\}_n$ and $\{U^n\}_n$, i.e.,
\[
    \liminf_{n\to\infty}\int_0^T\cA(f_t^n,U_t^n)\dx t\ge\int_0^T\cA(f_t,U_t)\dx t.
\]
\end{prop}
\begin{proof}
We first show that the total variation measure $\abs{U^n}$ is bounded on compact. We let $I \times K \subset \pra{0,T} \times \Rdiag$ be compacts. 
It is then relatively straightforward to see that
\begin{align}
  \abs{U^n} \bra{I \times K} \leq \int_I \abs{U_t^n}(K) \dx{t} \leq \int_I \int_K |\hat{U}_t^n(v, v_*)|\dx{\bra{f_t^n \otimes f_t^n}}(v,v_*) \dx{t} \, ,
\end{align}
where for the last inequality we have used finiteness of the action and the result of~\cref{prop:action}, which states that $U_t^n$ has a density with respect to $f_t^n\otimes f_t^n$. Upon applying the Cauchy--Schwartz inequality, we obtain the following bound
\begin{align}
    \label{eq:TV-bound}
    \begin{split}
  \abs{U^n} \bra{I \times K} & \leq \bra*{\int_I \int_K |\hat{U}_t^n(v,v_*)|^2 \sigma_e(\abs{v-v_*}) \dx{\bra{f_t^n \otimes f_t^n}}(v,v_*) \dx{t}}^{\frac12}\\ 
  &\qquad \times \bra*{\int_I \int_K \frac{ \dx{\bra{f_t^n \otimes f_t^n}}(v,v_*)}{\sigma_e(\abs{v-v_*})} \dx{t}}^{\frac12} \\
  & \leq  \bra*{\frac{1-e}{2} C_K |I|}^{\frac12} ,
  \end{split}
\end{align}
where $C_K=C \sup_{(v,v_*)\in K} \sigma_e(\abs{v-v_*})^{-1}< \infty$ with $C$ as in~\eqref{ass:compact:finite:action}, since $\sigma_e$ is continuous and positive on $\Rdiag$. 
Since \(I\times K\) was arbitrary, it is clear from the above estimate that we can obtain uniform local control on the total variation of the measures $U^n \in \cM\bra{\pra{0,T} \times\Rdiag}$. Thus by Prokhorov's theorem there exists a measure $U \in \cM\bra{\pra{0,T} \times\Rdiag}$ such that
$U^n \to^c U$, i.e., tested against $C_c([0,T]\times\Rdiag)$.

We now note that $U \in \cM_{\loc}(\pra{0,T} \times \Rdiag)$ can be disintegrated with respect to the Lebesgue measure on $\pra{0,T}$. Indeed, consider for any compact set, $K\subset \Rdiag$, the measure $\lambda^K:= \pi_{\#}^K U \in \cM\bra{\pra{0,T}}$, where $\pi^K: \pra{0,T} \times K \to \pra{0,T} $ is the projection map defined as $\pi^K(t,x):=t$, for $x\in K$. By the definition of the pushforward we have for any measurable $I\subset [0,T]$ from~\eqref{eq:TV-bound} the estimate
\begin{align}
    \lambda^K(I) = U(I \times K) \leq \bra*{\frac{1-e}{2} C_K |I|}^{\frac12}. 
\end{align}
Thus, $\lambda^K$ is absolutely continuous with respect for the Lebesgue measure on $I$, for any $K\subset \Rdiag$ compact. Additionally, for any 
$\varphi \in C_c(\pra{0,T}\times \Rdiag)$ choose $K\subset \Rdiag$ such that $\supp \varphi \subset [0,T]\times K$. By the disintegration theorem, cf. \cite[Theorem 5.3.1]{AGS08}, we have the existence of a family $\set{\mu_t^K}_{t\in [0,T]}$ such that $\dx U = \dx \mu_t^K \dx \lambda^K$. In particular
\begin{align}
    \int_0^T &\iint_{\Rdiag} \varphi(t,v,v_*)\sigma_e(|v-v_*|)\dx{U}(t,v,v_*)\\ 
    &=\int_0^T \bra*{\int_{\set{t} \times \Rdiag} \varphi(t,v,v_*)\sigma_e(|v-v_*|) \dx{\mu_t^K}(v,v_*)} \dx \lambda^K(t)\\
    &=\int_0^T \int_{\Rdiag} \varphi(t,v,v_*)\sigma_e(|v-v_*|) \dx{U_t^K}(v,v_*) \dx t, 
\end{align}
where $U_t^K:=  \frac{\dx{\lambda^K}}{\dx t} \mu_t^K$ and $\mu_t^K \in \cM(K)$ is the parametrised family of measures arising from the disintegration theorem.

We readily observe that integrating~\eqref{eq:GCE:weak_form} over $[t_1,t_2]$ gives for any $\psi\in C_c^1(\Rdiag)$
\begin{equation}\label{eq:W1:comp} \begin{split}
    \bigg|&\int_{\R} \psi(v) \dx{f_{t_1}^n}(v)- \int_{\R} \psi(v) \dx{f_{t_2}^n}(v)\bigg| \leq \int_{t_1}^{t_2} \iint_{\Rdiag} \abs*{\tilde{\nabla} \psi(v,v_*)} \sigma_e(|v-v_*|)  \dx{|U_t^n|}\dx{t}\\
    &\leq \int_{t_1}^{t_2} \iint_{\Rdiag} \abs*{\tilde{\nabla} \psi(v,v_*)} \sigma_e(|v-v_*|) \abs*{\hat{U}_t^n(v, v_*)}\dx{\bra{f_t^n \otimes f_t^n}}(v,v_*)\dx t \\
    &\leq \int_{t_1}^{t_2} \cA(f_t^n,U_t^n)^{\frac12} \bra*{ \iint_{\Rdiag} \abs*{ \tilde{\nabla} \psi(v,v_*)}^2 \sigma_e(|v-v_*|) \dx{\bra{f_t^n \otimes f_t^n}}(v,v_*)}^{\frac12} \dx t \\
    &\leq  \bra*{\frac{1-e}{4}}^{\frac12} \int_{t_1}^{t_2} \cA(f_t^n,U_t^n)^{\frac{1}{2}} \bra*{ \iint_{\Rdiag} \bra*{\psi'(v)-\psi'(v_*)}^2 \bra*{ |v| + |v_*|} \dx f_t^n(v) \dx f_t^n(v_*) }^{\frac12} \dx{t} \\
    &\leq C \norm{ \psi'}_{\infty} | t_2-t_1|^{\frac12}, 
\end{split}\end{equation}
according to \cref{eq:TV-bound}, having used the definition of $\sigma_e$, cf.  \eqref{eq:sigma_e} and applied the stability of the first moment~\eqref{e:m1:action} from Lemma~\ref{lem:finite_action:GCE}, which also ensures that $(f_t^n, U_t^n)_{t\in [0,T]}\in \GCE(f_0^n,f_T^n)$. Passing to the supremum in $\psi$ among all Lipschitz functions with Lipschitz constant 1, we recover the $1/2$-Hölder continuity in the 1-Wasserstein distance, i.e.,
\begin{align}
    d_1(f_{t_2}^n, f_{t_1}^n) \leq C |t_2-t_1|^{\frac12},
\end{align}
uniformly in $n\in \N$. An application of the generalised Arzela-Ascoli theorem concludes the proof of convergence of the densities, see \cite[Section 3]{AGS08}. In particular, we have that the limiting curve is absolutely continuous in time with values in probability measures and hence $(f_t,U_t)_{t\in [0,T]}\in \GCE(f_0,f_T)$. 
Finally, the lower semicontinuity property is a consequence of \cref{prop:lsc-action}.
\end{proof}

\subsection{The collision metric}\label{sec:metric}

In this section, we define and prove properties for an extended metric coming  from the nonlocal-local continuity equation. We start with the definition of the collision transportation cost.

\begin{defn}\label{def:metric}
Let $\mu_0,\mu_1\in\cP(\R)$. The \emph{collision transportation cost} is defined by
\begin{equation}
    \label{eq:def-metric}
    d_{\cA}(\mu_0, \mu_1)^2 := \inf\left\{\int_0^1\cA(f_t,U_t)\,\dx{t}: (f_t,U_t)_{t\in [0,1]}\in\GCE(\mu_0,\mu_1)\right\}.
\end{equation}
\end{defn}
Note that the minimisation problem above is well defined as consequence of the direct method of calculus of variations by means of \cref{prop:compactCRE}, whenever the action is bounded, i.e.,  $\int_0^1\cA(f_t,U_t)\,\dx{t}<\infty$. Moreover, by observing that $\alpha$ defined in~\eqref{eq:alpha-fun} is $2$-homogeneous in the second variable, we can apply the same reparametrisation argument used in \cite[Theorem 5.4]{DolbeaultNazaretSavare2009} to obtain the following result.
\begin{lem}[Reparametrisation]\label{lem:reparametrisation}
For any $T>0$, $\mu_0,\mu_1\in\cP(\R)$ it holds
\begin{align}
    d_{\cA}(\mu_0,\mu_1)=\inf\left\{\int_0^T \cA(f_t,U_t)^\frac12\,\dx{t}: (f_t,U_t)_{t\in [0,T]}\in\GCE_T(\mu_0,\mu_1)\right\}.
\end{align}
\end{lem}
In the following proposition we see under which conditions the infimum in \cref{eq:def-metric} is a minimum.
\begin{prop}
    \label{prop:metric:min}
    Let $\mu_0,\mu_1\in\cP(\R)$ such that $d_\cA:=d_\cA(\mu_0,\mu_1)<+\infty$. Then the infimum in \cref{eq:def-metric} is attained by a curve $(f_t,U_t)_{t\in [0,1]}\in\GCE(\mu_0,\mu_1)$ such that 
$$
    \cA(f_t,U_t)=d_\cA^2(\mu_0,\mu_1),
$$ 
for a.e. $t\in[0,1]$. Such a curve is a constant speed geodesic for $d_\cA$, i.e.,
\[
    d_\cA(f_s,f_t)=|t-s|d_\cA(\mu_0,\mu_1),
\]
for all  $s,t\in[0,1]$.
\end{prop}
\begin{proof}
If $d_\cA$ is finite, which holds when $\int_0^1 \cA(f_t,U_t) \dx{t} <\infty$ for some $(f_t,U_t)_{t\in [0,1]} \in \GCE(\mu_0,\mu_1)$, the infimum in \cref{eq:def-metric} is attained as a consequence of \cref{prop:compactCRE} by means of the direct method of calculus of variations. Thus, there exists a minimising curve $(f_t^*,U_t^*)_{t\in [0,1]}\in\GCE(\mu_0,\mu_1)$. By the reparametrisation result in \cref{lem:reparametrisation} and the Jensen's inequality, we obtain
\[
    \int_0^1\cA(f_t^*,U_t^*)^\frac12\dx{t} \ge d_\cA(\mu_0,\mu_1) = \left(\int_0^1\cA(f_t^*,U_t^*) \dx{t}\right)^{\frac12} \ge \int_0^1 \cA(f_t^*,U_t^*)^\frac12 \dx{t},
\]
whence $d_\cA^2(\mu_0, \mu_1)=\cA(f_t^*,U_t^*)$, for almost every $t\in[0,1]$. Moreover, we obtain
\[
    d_\cA(f_s,f_t) = \int_s^t \cA(f_r^*,U_r^*)^\frac12\dx{r} = |t-s|d_\cA(\mu_0,\mu_1),
\]
for all $s,t\in[0,1]$, which concludes the proof.
\end{proof}
Given the preservation of the centre of mass and the stability of the first moment along curves of finite action implied by Proposition~\ref{prop:GCE:centre_mass}, it makes sense to restrict the collision transport cost to certain subspaces.
Let us note the metric $d_\cA$ can be compared with $d_1$, the $1$-Wasserstein distance.
\begin{prop}[\label{prop:comparison-W1}Comparison with $d_1$]
Let $\mu_0,\mu_1\in\cP_1(\R)$. There exists a constant $C=C(e)$ such that
\[
    d_1(\mu_0,\mu_1) \le C \bra*{m_1(\mu_0) + d_{\cA}(\mu_0,\mu_1)} d_\cA(\mu_0,\mu_1) .
\]
\end{prop}
\begin{proof}
The proof is obtained along the lines of the estimate~\eqref{eq:W1:comp}, and using \eqref{e:m1:action}.
\end{proof}
\begin{thm}\label{thm:metric}
The collision transport cost defined in~\eqref{eq:def-metric} is an extended metric on $\cP(\R)$. The map $(\mu_0,\mu_1)\mapsto d_\cA(\mu_0,\mu_1)$ is lower semicontinuous with respect to the convergence in $\cP(\R)$. Moreover, the topology induced by $d_\cA$ is stronger then the $d_1$-topology.
\end{thm}
\begin{proof}
Let us assume that $d_\cA(\mu_0,\mu_1)=0$. By Proposition \ref{prop:metric:min} there exists a curve $(f_t,U_t)_{t\in [0,T]}\in\GCE(\mu_0,\mu_1)$ such that $\cA(f_t, U_t)=0$ for a.e. $t\in[0,1]$, which implies $ U_t=0$ for a.e. $t\in[0,1]$. Thus, from \cref{eq:GCE:weak_form} we obtain $\mu_0=\mu_1$. The opposite implication is trivial. The symmetry of $d_\cA$ follows from the fact that $\alpha(\cdot, u)=\alpha(\cdot, -u)$. In order to prove the triangle inequality we notice that solutions to $\GCE$ can be concatenated. Indeed, if $(f^i, U^i)\in\GCE_{T_i}(\mu_0^i,\mu^i_{T_i})$ for $i=1,2$ such that $\mu_{T_1}^1=\mu^2_0$, then
\begin{equation}
f_t:=\begin{cases}
    f_t^1 \quad &\mbox{if } 0\le t\le T_1\\
    f_{t-T_1}^2 \quad &\mbox{if } T_1\le t\le T_1+T_2
    \end{cases};\quad
 U_t:=\begin{cases}
     U_t^1 \quad &\mbox{if } 0\le t\le T_1\\
     U_{t-T_1}^2 \quad &\mbox{if } T_1\le t\le T_1+T_2
    \end{cases}
\end{equation}
belongs to $\GCE_{T_1+T_2}(\mu_0^1,\mu_{T_2}^2)$ by using \cref{eq:continuous-repr}. This observation and \cref{lem:reparametrisation} imply the triangle inequality. The lower semicontinuity property is a consequence of \cref{prop:compactCRE}, while \cref{prop:comparison-W1} gives that the topology induced by $d_\cA$ is stronger than that of $d_1$.
\end{proof}
Let us recall the definition of absolutely continuous curves in a metric space. A curve $[0,T]\ni t\mapsto f_t\in \cP(\R)$ is said to be $2$-\textit{absolutely continuous} with respect to $d_{\cA}$ if there exists $m\in L^2(0,T)$ such that
\begin{align}\label{eq:def-abs-cont}
    d_{\cA}(f_{t_0},f_{t_1})\le\int_{t_0}^{t_1}m(t)\dx{t}, \quad \mbox{for all} \quad 0<t_0\le t_1<T.
\end{align}
In this case, we write $f\in\AC(0,T;(\cP(\R),d_{\cA}))$. For any $f\in\AC(0,T;(\cP(\R),d_{\cA}))$ the quantity
\begin{align}
    |f'|(t)=\lim_{h\to0}\frac{d_{\cA}(f_{t+h},f_t)}{h}
\end{align}
is well-defined for a.e. $t\in[0,T]$ and is called \textit{metric derivative} of $f$ at $t$. Moreover, the function $t\to |f'|(t)$ belongs to $L^2(0,T)$ and it satisfies $|f'|(t)\le m(t)$ for a.e. $t\in[0,T]$, i.e., $f'$ is the minimal integrand satisfying \eqref{eq:def-abs-cont}. The length of a curve $f\in\AC(0,T;(\cP(\R),d_{\cA}))$ is defined by $L(f):=\int_0^T|f'|(t)\dx{t}$.

Given the above results we can easily obtain the following characterisation, as in \cite[Theorem 5.17]{DolbeaultNazaretSavare2009}. The proof is then omitted.
\begin{prop}[Metric velocity]\label{prop:ACcurves}
A curve $\{f_t\}_{t\in[0,T]}\subset \cP(\R)$ belongs to the space $\AC(0,T;(\cP(\R),d_{\cA}))$ if and only if there exists a family of flux $\{U_t\}_{t\in[0,T]}$ such that $\set{(f_t,U_t)}_{t\in [0,T]}\in \GCE_T$ with 
\[
\int_0^T\cA(f_t,U_t)^\frac12 \dx{t} < \infty.
\]
In particular, $\dx U_t(v,v_*) = \hat U_t(v,v_*) \dx(f_t\otimes f_t)(v,v_*)$ for a measurable family $\hat U : [0,T]\times \Rdiag \to \R$.
In this case, the metric derivative is bounded as in $|f'|^2(t)\le\cA(f_t,U_t)$ for a.e. $t\in[0,T]$. In addition, there exists a unique $\{\tilde{U}_t\}_{t\in[0,T]}$ such that $(f_t, \tilde U_t)_{t\in [0,T]}\in \GCE_T$ and
    \begin{align}
        \label{eq:metric-dev-action}
        |f'|^2(t)=\cA(f_t,\tilde{U}_t), \qquad\text{for a.e. } t\in[0,T].
    \end{align}
\end{prop}
\begin{cor}[Tangent space]
Let $\set{(f_t,U_t)}_{t\in [0,T]} \in \GCE_T$ such that the curve $f\in\AC(0,T;(\cP(\R),d_{\cA}))$. The flux $U$ satisfies \eqref{eq:metric-dev-action} if and only if $U_t\in T_f\cP(\R)$ for a.e. $t\in[0,T]$, where
\begin{equation}\label{eq:tang-space}
    \begin{split}
        T_f\cP(\R)=\bigl\{&U\in\cM(\Rdiag):\cA(f,U)<\infty, \,  \cA(f,U)\le\cA(f,U+w),\\
        &\mbox{for any  }w\in\cM(\Rdiag), \mbox{ s.t. } \tilde\nabla\cdot w=0\bigr\}.
    \end{split}
\end{equation}
\end{cor}
\begin{proof}
According to Proposition \ref{prop:ACcurves} the metric derivative satisfies $|f'|^2(t)\le\cA(f_t,U_t)$ for a.e. $t\in[0,T]$. Therefore, the only flux satisfying \eqref{eq:metric-dev-action} is that of minimal action.
Let $t\in[0,T]$ such that $\cA(f_t,U_t)<+\infty$. As proved in Proposition \ref{prop:antisymm-flux}, the flux, $\tilde{U}_t$, of minimal action has to be antisymmetric, $\tilde{U}_t\in \cM^{\mathrm{as}}(\Rdiag)$, and by assumption satisfy the nonlocal-local continuity equation. In particular, 
\begin{align}\label{eq:min-flux}
\tilde{U}_t=\argmin_{U\in\cM^{\mathrm{as}}(\Rdiag)}\{\cA(f_t,U):\tilde \nabla\cdot U_t=\tilde\nabla\cdot U\}.
\end{align}
Note that the set $\{U\in\cM^{as}(\Rdiag):\tilde \nabla\cdot U_t=\tilde\nabla\cdot U\}$ is closed with respect to the weak-$^*$ convergence, and sublevel sets of the functional  $\cM^{\mathrm{as}}(\Rdiag)\ni U \mapsto \cA(f,U)$, for any $f \in \cP(\R)$, are locally weakly-$^*$ relatively compact by arguing as in Proposition \ref{prop:compactCRE}, since for any compact set $K\subset \Rdiag$ it holds
\[
|U|(K)\le \cA(f_t,U)^\frac12 \sup_{K}\sigma_e(|v-v_*|)^{-1}.
\]
Moreover, note that the functional $\cM^{\mathrm{as}}(\Rdiag) \ni U \mapsto \cA(f,U)$, for any $f \in \cP(\R)$, is strictly convex according to Lemma \ref{prop:action}. Therefore, the flux in \eqref{eq:min-flux} is uniquely determined.
\end{proof}
In the previous corollary we have a Lagrangian formulation of the tangent space $T_f\cP(\R)$, which can be further characterised in terms of tangent velocity fields. 
\begin{prop}
    Let $f\in\cP(\R)$. Then, it holds that $U\in T_f\cP(\R)$ if and only if $U\in\cM(\Rdiag)$ such that $\cA(f,U)<\infty$ and, for a measurable $\hat{U}:\Rdiag\to\R$, it holds
\[
\hat{U}\in\overline{\{\tilde \nabla\varphi:\varphi \in C_c^\infty(\R)\}}^{L^2(\Rdiag,\sigma_e\dx{(f\otimes f)})}
\] 
\end{prop}
\begin{proof}
If the action $\cA(f,U)<\infty$, Lemma \ref{prop:action} provides the existence of a measurable $\hat{U}:\Rdiag\to\R$ such that $\dx{U}(v,v_*)=\hat{U}(v,v_*)\dx{(f\otimes f)}(v,v_*)$, for any $(v,v_*)\in\Rdiag$, whence
\[
\cA(f,U)=\iint_{\Rdiag}|\hat{U}(v,v_*)|^2\sigma_e(|v-v_*|)\dx{(f\otimes f)}(v,v_*)=\|\hat{U}\|^2_{L^2(\sigma_e\dx{(f\otimes f)})}.
\]
As consequence of the above relation between $U$ and $\hat{U}$, the nonlocal divergence $\tilde\nabla\cdot U$ can be re-written in terms of $\hat{U}$, for any $\varphi\in C_c^\infty(\R)$, as
\[
\iint_{\Rdiag}\tilde\nabla\varphi(v,v_*)\sigma_e(|v-v_*|)\dx U(v,v_*)=\iint_{\Rdiag}\tilde\nabla\varphi(v,v_*)\hat{U}(v,v_*)\sigma_e(|v-v_*|)\dx{(f\otimes f)}(v,v_*).
\]
Thus, the characterisation \eqref{eq:tang-space} can be equivalently stated as
\begin{align}
    \iint_{\Rdiag}|\hat{U}|^2\sigma_e(|\cdot -\cdot|)\dx{(f\otimes f)}\leq \iint_{\Rdiag}|\hat{U}+W|^2\sigma_e(|\cdot-\cdot|)\dx{(f\otimes f)},
\end{align}
for all $W\in L^2(\Rdiag,\sigma_e\dx{(f\otimes f)})$ such that
\begin{align}
    \iint_{\Rdiag}\tilde\nabla\varphi(v,v_*)W(v,v_*)\sigma_e(|v-v_*|)\dx{(f\otimes f)}(v,v_*)=0 \qquad\text{for all $\varphi\in C_c^\infty(\R)$.}
\end{align}
Therefore, $\hat{U}$ belongs to the closure of $\{\tilde\nabla\varphi:\varphi \in C_c^\infty(\R)\}$ in $L^2\bra*{\Rdiag,\sigma_e\dx{(f\otimes f)}}$.
\end{proof}

\section{The aggregation equation in a new light}
This section focuses on the aggregation equation \eqref{eq:AggregationEquation}, with a cubic interaction potential~\eqref{e:def:W}.
As discussed in Section \ref{subsec:derivation-aggr-eq}, \eqref{eq:AggregationEquation} can be formally derived from the inelastic spatially homogeneous Boltzmann equation by Taylor-expanding the test function in its weak formulation. In this process, we notice that the collision kernel obtained from the cubic interaction, $W$, is precisely the modulus function. This suggests that we interpret~\eqref{eq:AggregationEquation} as \textit{nonlocal-local continuity equation}, as explained in Section \ref{sec:GCE}, driven by the potential obtained from the kinetic energy~\eqref{eq:KineticEnergy}.

More precisely, in this Section, we consider the \eqref{eq:gce} driven by the kinetic energy~\eqref{eq:KineticEnergy}.
In addition to the definition of weak solutions to~\eqref{eq:gce} (see Definition~\ref{def:CRE}), we require the curve to have finite kinetic energy, which is a natural requirement.
\begin{defn}[Weak solution]\label{def:weak-sol-gcae}
A curve $\set{f_t}_{t\in [0,T]} \subset \cP_2^{\cm}(\R)$ is a \emph{weak solution} to~\eqref{eq:AggregationEquation} if, for the flux $\set{U_t^\cE}_{t\in [0,T]} \subset \cM(\Rdiag)$ given by
\begin{align}\label{eq:gen-cae}
    \dx U_t^\cE(v,v_*)=-\tilde\nabla\frac{\delta \cE}{\delta f}(v,v_*)\dx (f_t\otimes f_t)(v,v_*),
\end{align}
the pair $\set{(f_t,U_t^\cE)}_{t\in[0,T]}$ satisfies the \emph{nonlocal-local continuity equation}~\eqref{eq:gce} in the sense of Definition \ref{def:CRE}.
\end{defn}
In order to achieve a new gradient flow formulation of the  equation above as steepest descent of the kinetic energy with respect to the collision metric defined in Section \ref{sec:metric}, we follow \cite{AGS08} and use the concept of curve of maximal slope with respect to a specific strong upper gradient, which is the square root of the dissipation functional, cf. \eqref{eq:dissipation} below. 
To motivate this, we consider the decay of the kinetic energy along a curve \(f\in \AC([0,T];(\cP(\R),d_{\cA}))\) which is a solution of the nonlocal-local continuity equation~\eqref{def:CRE}, i.e., there exists a flux $\dx U_t = \hat U_t \dx (f\otimes f)$ such that the pair $\{(f_t,U_t)\}_{t\in [0,T]}$ is a weak solution in the sense of~\cref{def:CRE}. Formally applying the chain rule, we have
\begin{align}\label{eq:formal:chain-rule}
    \cE(f_T)- \cE(f_0) &= \int_0^T \!\iint_{\Rdiag} \! \tilde{\nabla}\frac{\delta \cE}{\delta f}(v,v_*) \hat U_t(v,v_*) \sigma_e(|v-v_*|) \dx (f \otimes f)(v, v_*) \dx{t}.
\end{align}
After an application of Young's inequality to both the inner integrals with weight $\sigma_e \dx(f \otimes f)$, we observe
\begin{align}
    \int_0^T \!\iint_{\Rdiag} &\! \tilde{\nabla}\frac{\delta \cE}{\delta f}(v,v_*) \hat U_t(v,v_*) \sigma_e(|v-v_*|) \dx (f \otimes f)(v,v_*)\dx{t}, \\
    &\geq - \frac{1}{2} \int_0^T \abs*{\hat U_t(v,v_*)}^2 \sigma_e(\abs*{v-v_*}) \dx (f_t\otimes f_t)(v,v_*) \dx t   \\
    &\quad - \frac{1}{2} \int_0^T \iint_{\Rdiag} \abs*{\tilde{\nabla}\frac{\delta \cE}{\delta f}(v,v_*)}^2 \sigma_e(\abs*{v-v_*}) \dx (f_t\otimes f_t)(v,v_*) \dx t \\
    &= -\frac12 \int_0^T \cA(f_t, U_t)  \dx t - \frac12 \int_0^T \cD(f_t) \dx t,
\end{align}
where the dissipation is defined by
\begin{align}
    \label{eq:dissipation}
    \cD(f):= \iint_{\Rdiag} \abs*{v - v_*}^2 \sigma_e(\abs*{v-v_*}) \dx{\bra*{f \otimes f}}(v, v_*),
\end{align}
cf. also~\eqref{eq:KineticEnergyDissipation}, in the context of the formal derivation. 
This motivates our definition of gradient flow solutions as curves~$f\in \AC([0,T];(\cP_2^{\cm}(\R),d_{\cA}))$ in the zero locus of the \emph{De Giorgi} functional
\begin{align}\label{def:DeGiorgino}
    \cG_T(f) := \cE(f_T)- \cE(f_0) + \frac12 \int_0^T \cA(f_t, U_t)  \dx t + \frac12 \int_0^T \cD(f_t) \dx t .
\end{align}
Based on the preceding computations we introduce our notion of gradient flow solutions as curves of maximal slope. 
\begin{defn}[Curves of maximal slope]
\label{def:CurvesMaxSlope}
  A curve $f \in \AC([0,T],(\cP_2^{\cm}(\R),d_{\cA}))$ is a curve of maximal slope if $\cG_T(f)=0$.
\end{defn}
In order to show that weak solutions to \eqref{eq:gen-cae} are curves of maximal slope and  to mathematically justify the definition of the De Giorgi functional  \eqref{def:DeGiorgino}, we need to rigorously derive the chain rule in~\eqref{eq:formal:chain-rule}.
In particular, the chain rule implies that the square root of the dissipation functional $\cD$, defined in \eqref{eq:dissipation}, is a strong upper-gradient for $\cE$ with respect to the extended metric $d_\cA$ (cf.~\cite[Definition 1.2.1]{AGS08}).

\subsection{The chain rule and characterisation of weak solutions}

\begin{lem}[Stability and chain rule]\label{lem:crule}
Let $T>0$ and $\set{(f_t,U_t)}_{t\in [0,T]} \in \GCE_T(\mu_0)$ for some $\mu_0  \in \cP_2^{\cm}(\R)$. Assume that
\begin{align}
    \int_0^T  \cA(f_t, U_t)^\frac12 \dx{t} < \infty, \qquad \text{and} \qquad \int_0^T \! \cA(f_t, U_t)^\frac12 \cD(f_t)^\frac12 \dx{t} <\infty \, ,
\label{eq:ADbound}
\end{align} 
where $\cA: \cP(\R) \times \cM(\Rdiag) \to (-\infty,+\infty]$ is the action, as defined in~\cref{def:action-functional}, and $\cD: \cP(\R) \to (-\infty,+\infty]$ is the dissipation defined in~\eqref{eq:dissipation}.

Then, the following properties hold:
\begin{tenumerate}
\item $\sup_{t \in [0,T]} \cE(f_t) < \infty$.
\item For any $0\leq s \leq t \leq T$
\begin{align}
    \cE(f_t)- \cE(f_s) &= \int_s^t \!\iint_{\Rdiag} \! \tilde{\nabla}\frac{\delta \cE}{\delta f}(v,v_*) \sigma_e(|v-v_*|) \dx U_\tau(v,v_*) \, \dx{\tau}.
\end{align}
\end{tenumerate}
\end{lem}
\begin{proof}

We define a globally Lipschitz approximation of $|v|^2/2$ which we can use as a test function in the weak formulation of \eqref{eq:gce} by Remark~\ref{rem:w1inf}. Let
\begin{align}
    \label{eq:really_awesome_truncation}
    \varphi_R(v) :=
    \left\{
    \begin{array}{ll}
         \ds v^2/2, &  v\in [0,R],\\
         R^2/2 + R (v-R), & v \in [R, \infty),
    \end{array}
    \right.
\end{align}
and extend it to $\R$ by setting $\varphi_R(v) = \varphi_R(-v)$ for $v\in (-\infty, 0)$. 
Note, that this choice of test function also satisfies the following condition
\begin{align}
    \abs*{\frac{\varphi_R'(v) - \varphi_R'(v_*)}{v-v_*}} \leq 1, 
\end{align}
which we will exploit in the subsequent analysis. For any weak solution of~\eqref{eq:gce}, $\set{(f_t, U_t)}_{t\in[0,T]}$, there holds~\eqref{eq:GCE:weak_form}, i.e.,
\begin{align}
    \int_\R \varphi(v) \dx \tilde f_{T}(v) - \int_\R \varphi(v) \dx \tilde f_{0}(v) = \int_0^T \iint_{\Rdiag} \tilde\nabla \varphi(v,v_*) \sigma_e(|v-v_*|) \dx U_t(v,v_*)\dx{t},
\end{align}
for any regular test function, $\varphi \in C_c^1(\R)$. In particular, choosing $\varphi = \varphi_R$, with $\varphi_R$ as in  \eqref{eq:really_awesome_truncation}, we have
\begin{align}
    \label{eq:weak_form_truncd}
    \int_\R \varphi_R(v) \dx \tilde f_{T}(v) \!-\! \int_\R \varphi_R(v) \dx \tilde f_{0}(v) = \int_0^T\!\! \iint_{\Rdiag} \tilde\nabla \varphi_R(v,v_*) \sigma_e(|v-v_*|) \dx U_t(v,v_*)\dx t,
\end{align}
where we can estimate the right-hand side as follows:
\begin{align}
    \MoveEqLeft\int_0^T \iint_{\Rdiag} \tilde \nabla \varphi_R (v,v_*)  \sigma_e(|v-v_*|) \dx{U_t}(v,v_*)\dx{t} \\
    &= \int_0^T\iint_{\Rdiag} \tilde \nabla \varphi_R (v,v_*) \sigma_e(|v-v_*|) \hat{U}_t(v,v_*) \dx (f_t\otimes f_t)(v,v_*)\dx{t} \\
    &\leq \int_0^T \bra*{\iint_{\Rdiag} \sigma_e(|v-v_*|) |\hat U_t(v,v_*)|^2  \dx(f_t\otimes f_t)(v,v_*)}^{\frac12}\\
    &\qquad \times \bra*{\iint_{\Rdiag} \abs*{\tilde \nabla \varphi_R}^2 \sigma_e(|v-v_*|) \dx (f_t\otimes f_t)(v,v_*)}^{\frac12} \dx t\\
    &=\int_0^T \cA(f_t, U_t)^{\frac12}  \bra*{\iint_{\Rdiag} \abs*{\tilde \nabla \varphi_R}^2 \sigma_e(|v-v_*|) \dx (f_t \otimes f_t)(v,v_*)}^{\frac12} \dx t\\
    &=\int_0^T \cA(f_t, U_t)^{\frac12}  \bra*{\iint_{\Rdiag} \abs*{\frac{\varphi_R'(v) - \varphi_R'(v_*)}{v-v_*}}^2 |v-v_*|^2 \sigma_e(|v-v_*|) \dx (f_t \otimes f_t)(v, v_*)}^{\frac12} \!\!\!\dx t\\
    &\leq \int_0^T \cA(f_t, U_t)^\frac12 \cD(f_t)^\frac12\dx t.
\end{align}
Hence, the right-hand side is uniformly integrable and due to the pointwise convergence of $\varphi_R$ we may pass to the limit $R\to\infty$ in the weak form,  \eqref{eq:weak_form_truncd}, due to Lebesgue's dominated convergence theorem. Hence we get
\begin{align}
    \cE(f_T)- \cE(f_0) &= \int_0^T \!\iint_{\Rdiag} \! \tilde{\nabla}\frac{\delta \cE}{\delta f}(v,v_*) \sigma_e(|v-v_*|) \dx \hat U_t(v,v_*) \, \dx f(v) \dx f(v_*)\dx{t},
\end{align}
as claimed in the statement.

As the test function $\varphi_R$ in~\eqref{eq:really_awesome_truncation} has linear growth at infinity, we can use it in the weak formulation in~\eqref{eq:GCE:weak_form} by Remark~\ref{rem:extension:test-functions}, i.e.,
\begin{align}    \label{eq:energy_dissipation_rig}
\begin{split}
    \frac{\dx{}}{\dx{t}} \int_{\R} \varphi_R(v) \dx{f_t}(v)  &= - \frac{1-e}{4}\iint_{\Rdiag}  |v-v_*| \tilde \nabla\varphi_R (v,v_*)  (v_* - v) \dx{f_t}(v) \dx{f_t}(v_*)\, .
\end{split}
\end{align}
 By expanding the definition of $\tilde\nabla\varphi_R$ from~\eqref{eq:nl_grad} and using the short-hand notation
\begin{align}
    \dx g(v,v_*) := |v-v_*| \bra*{\varphi_R'(v_*)-\varphi_R'(v)}  (v_* - v) \dx (f\otimes f)(v, v_*),
\end{align}
we have
\begin{align}
     \frac{\dx{}}{\dx{t}} \int_{\R} \varphi_R(v) \dx{f_t}(v)  &  = - \frac{1-e}{4}\left(\mathcal{I}_1 + \ldots  + \mathcal{I}_9\right),
\end{align}
with 
\begin{align}
     \mathcal{I}_1 = \frac12 \int_{-\infty}^{-R} \int_{-\infty}^{-R} \dx g(v,v_*), \quad
     \mathcal{I}_2 = \frac12 \int_{-\infty}^{-R} \int_{-R}^R \dx g(v,v_*), \quad \mathcal{I}_3 = \frac12 \int_{-\infty}^{-R} \int_{R}^\infty \dx  g(v,v_*),
\end{align}
and
\begin{align}
     \mathcal{I}_4 = \frac12 \int_{-R}^{R} \int_{-\infty}^{-R} \dx g(v,v_*), 
     \quad 
     \mathcal{I}_5 = \frac12 \int_{-R}^{R} \int_{-R}^R \dx g(v,v_*),
     \quad
     \mathcal{I}_6 = \frac12 \int_{-R}^{R} \int_{R}^\infty \dx g(v,v_*),
\end{align}
as well as
\begin{align}
     \mathcal{I}_7 = \frac12 \int_{R}^{\infty} \int_{-\infty}^{-R}  \dx g(v,v_*),
     \quad
     \mathcal{I}_8 = \frac12 \int_{R}^{\infty} \int_{-R}^R \dx g(v,v_*),
     \quad
     \mathcal{I}_9 = \frac12 \int_{R}^{\infty} \int_{R}^\infty \dx  g(v,v_*).
\end{align}
It is immediately clear that $\mathcal{I}_1 = \mathcal{I}_9 = 0$, as $\tilde \nabla \varphi_R$ vanishes in the respective ranges for $v,v_*$, whence $g(v,v_*)=0$. It is easy to verify that $\mathcal{I}_j \geq 0$, for $j\neq 5$.
We expand on the argument for $\mathcal{I}_2$ and note that arguments along similar lines will allow us to treat the remaining terms. Indeed, 
\begin{align}
    \mathcal{I}_2 &=  \int_{-\infty}^{-R}\int_{-R}^R  |v - v_*| (v_* + R) (v_*-v) \dx{(f_t \otimes f_t)}(v,v_*) \geq 0,
\end{align}
since $v_* \geq -R \geq v$ in the domain of integration.
Substituting $\mathcal{I}_j\geq 0$, for $j\neq 5$, into  \eqref{eq:energy_dissipation_rig}, we get
\begin{align}
    \int_{\R} \varphi_R(v)\dx{f_t}(v) - \int_{\R} \varphi_R(v)\dx{f_s}(v) + \frac{1-e}{4} \int_s^t\int_{-R}^R\int_{-R}^R |v-v_*|^3 \dx{(f_t \otimes f_t)}(v,v_*) \leq 0,
\end{align}
having integrated in time. By the dominated convergence theorem and the finite initial kinetic energy, we obtain
\[
    \frac12 \int_{\R} |v|^2 \dx{f_t}(v) +   \frac{1-e}{4}  \int_0^t\iint_{\Rdiag} |v-v_*|^3 \dx{(f_t \otimes f_t)}(v,v_*) \leq \frac12 \int_{\R} |v|^2\dx{f_0}(v).
\]
\qedhere
\end{proof}
\begin{rem}
    \mbox{}
    \begin{enumerate}
        \item Let us highlight that the proof of the dissipation of the kinetic energy via the truncation argument using the test functions, $\varphi_R$, is absolutely independent of assumption \eqref{eq:ADbound}. Indeed, it is not too surprising that we require the kinetic energy to be dissipated along the aggregation equation regardless of the metric setting. In particular, any weak solution from Definition~\ref{def:weak-sol-gcae} satisfies \begin{equation}\label{eq:energy-dissipation}
          \cE(f_T) + \int_0^T \cD(f_t) \dx{t} \leq \cE(f_0) . 
        \end{equation}
        \item Note that the statement of the theorem is true for any absolutely continuous curve, namely $\set{f_t}_{t\in[0,T]} \in \AC([0,T];(\cP(\R),d_{\cA}))$ with $f_0\in \cP_2^{\cm}(\R)$ and $\int_0^T \cD(f_t)\dx t < \infty$. In this case the action is always bounded and implies the existence of an associated flux, using the characterisation of absolutely continuous curves stated in Proposition \ref{prop:ACcurves}. 
    \end{enumerate}
\end{rem}
As direct consequence of the chain rule we have $\cD^\frac12$ is a strong upper gradient with respect to the distance $d_\cA$ in the sense of~\cite[Definition 1.2.1]{AGS08}
\begin{cor}\label{cor:stron-upp-gr}
For any curve $f\in \AC([0,T];(\cP(\R),d_\cA))$ with $f_0\in \cP_2^{\cm}(\R)$ it holds
\begin{equation}
    |\cE(f_t)-\cE(f_s))|\leq \int_s^t \cD(f_r)^\frac12 |f'_r| \dx r\qquad \forall\, 0\leq s\leq t\leq T,
\end{equation}
that is $\cD^\frac12$ is a strong upper gradient for $\cE$.
\end{cor}
\begin{proof}
  Without loss of generality, we can assume $\int_s^t \cD(f_r)^\frac12 |f'_r| \dx r<\infty$, otherwise the claim is immediately true. The result follows from Lemma \ref{lem:crule} by applying Cauchy-Schwartz inequality and using the characterisation of absolutely continuous curves stated in Proposition \ref{prop:ACcurves}.
\end{proof}
We are now able to characterise weak solutions as curves of maximal slope in the sense of Definition \ref{def:CurvesMaxSlope}.
\begin{thm}[Weak solutions are curves of maximal slope]\label{thm:weak-curves}~\\
 A curve $f\in \AC\bra*{[0,T],(\cP_2^{\cm}(\R),d_{\cA})}$ is a weak solution to~\eqref{eq:AggregationEquation} 
 in the sense of Definition~\ref{def:weak-sol-gcae} if and only if $\mathcal{G}_T(f)=0$.
\end{thm}
\begin{proof}
Let $f$ be a weak solution in the sense of Definition \ref{def:weak-sol-gcae} with corresponding flux $U_t^\cE(v,v_*)$ given by~\eqref{eq:gen-cae}. 
It can be checked that $\cA(f_t,U_t^\cE)=\cD(f_t)$ and by the energy dissipation~\eqref{eq:energy-dissipation} also follows that $\cE(f_T) + \int_0^T \cD(f_t)\dx t \leq \cE(f_0) < \infty$. In particular, $\cE(f_T) - \cE(f_0) + \frac{1}{2}\int_0^T (\cA(f_t,U_t^\cE) +\cD(f_t)) \dx{t} \leq 0$, whence $\cG_T(f)\leq 0$ and $f\in \AC([0,T];(\cP(\R),d_{\cA}))$. Thus, by the chain rule Lemma~\ref{eq:formal:chain-rule}, we have that $\cG_T(f)\geq 0$. Hence, $\cG_T(f)=0$.

Let us now assume that $f\in \AC([0,T];(\cP(\R),d_{\cA}))$ satisfies $\cG_T(f)=0$.
According to Proposition \ref{prop:ACcurves} there exists a unique family $\{\dx U_t = \hat U_t \dx(f_t \otimes f_t)\}_{t\in[0,T]}$ such that $\set{(f_t, U_t)}_{t\in [0,T]}\in\GCE_T$ and  $\int_0^T\cA(f_t, U_t)\dx t<\infty$.
By the chain rule Lemma~\ref{lem:crule}, we obtain
\begin{align}
    0 &=\cG_T(f_t) = \cE(f_T) - \cE(f_0) + \frac12 \int_0^T \cA(f_t, U_t) \dx t + \frac12 \int_0^T\cD(f_t)\dx t\\
    &=\int_0^T\iint_{\Rdiag} \tilde \nabla \frac{\delta \cE}{\delta f} \sigma_e(|v_* - v|) \hat U_t(v,v_*) \sigma_e(|v_*-v|)\dx (f_t\otimes f_t)(v,v_*)\dx t \\
    &\qquad + \frac12 \int_0^T \iint_{\Rdiag} \bra*{ \abs*{\tilde \nabla \frac{\delta \cE}{\delta f}}^2 + \abs*{\hat U_t}^2} \sigma_e(|v_*-v|) \dx (f_t\otimes f_t)(v,v_*) \dx{t}\\
    &= \frac12 \int_0^T \iint_{\Rdiag} \abs*{\tilde \nabla \frac{\delta \cE}{\delta f} + \hat U_t}^2\sigma_e(|v_*-v|) \dx (f_t\otimes f_t)(v,v_*) \dx{t} .
\end{align}
Hence 
\[
    \hat U_t(v,v_*) = - \tilde \nabla \frac{\delta \cE}{\delta f}(v,v_*) = v - v_*,
\] 
which implies that $U_t =  U_t^\cE$, from~\eqref{eq:gen-cae}.
\end{proof}

To establish the existence of minimisers of the De Giorgi functional in~\eqref{def:DeGiorgino}, we have to prove lower semicontinuity of the dissipation.
\begin{prop}[Lower semicontinuity of the dissipation]\label{prop:dissipation:lsc}
  Let $\set{f^n}_{n\in \N } \subset \cP(\R)$ such that $f^n \to  f \in \cP(\R)$, then it holds
  \begin{equation}
      \liminf_{n\to \infty} \cD(f^n) \geq \cD(f) . 
  \end{equation}
\end{prop}
\begin{proof}
  We consider a cut-off away from the diagonal. Let $\varphi_R(r)\in C_c^1(\R)$ be such that $\varphi_R(r) = 1$ for $r\in [-R,R]$ and $\varphi_R(r)= 0$ for $\abs{r}\geq 2R$, then we have by positivity of the integrand in $\cD(f^n)$ the estimate
  \begin{equation}
      \cD(f^n) \geq \iint_{\Rdiag} \varphi_R(|v-v_*|) |v-v_*|^2 \sigma_e(|v-v_*|) \dx (f^n \otimes f^n)(v,v_*).
  \end{equation}
Hence, the proof is concluded by letting $n\to\infty$ first, and via monotone convergence for $R\to\infty$.
\end{proof}

\subsection{Stability and existence by particle approximation}

To discuss the existence of curves of maximal slope, we proceed by a strategy similar to showing existence of solutions to the aggregation equation by finite-dimensional approximations, cf.~\cite{CDFLS11}. %

Let us first summarise the given compactness and lower semicontinuity statements for the objects in the definition of the De Giorgi functional, cf. \eqref{def:DeGiorgino}, which provide the stability of curves of maximal slope in our setting.
By combining the lower semicontinuity of the action in Proposition~\ref{prop:lsc-action} and the lower semicontinuity of the dissipation in Proposition~\ref{prop:dissipation:lsc}, as well as noting that the kinetic energy~\eqref{eq:KineticEnergy} is lower semicontinuous with respect to narrow convergence due to the convexity of the integrand, we obtain the stability of curves of maximal slope.
\begin{thm}[Stability of curves of maximal slope]\label{thm:stability}
  Let the sequence $\set{f^n}_{n\in\N} \subset \AC([0,T],(\cP_2^{\cm}(\R),d_{\cA}))$ be such that $\sup_{n\in \N}\cG(f^n)< \infty$ and $\cE(f^n_0)\to \cE(f_0)$ with $f^n_0\to f_0$, then there exists some $f\in \AC([0,T],(\cP_2^{\cm}(\R),d_{\cA}))$ such that $f_t^n\to f_t$, for a.e.~$t\in [0,T]$ and
  \begin{equation}
      \liminf_{n\to \infty} \cG(f^n) \geq \cG(f) . 
  \end{equation}
\end{thm}
Based on this stability statement for curves of maximal slope we may now construct solutions devising an approximation by particles. Let us stress that existence of minimisers for $\cG_T$ can be shown by the direct method of calculus of variations. However, this does not provide that minima are zeros of $\cG_T$.
\begin{thm}[Existence by particle approximation]\label{lem:particle}
  For any $f_0 \in \cP_2^{\cm}(\R)$, that is $\cE(f_0)<\infty$, there exists a curve of maximal slope. 
  \label{lem:exCOMS}
\end{thm}
\begin{proof}
  The strategy is based on constructing a particle approximation of the initial measure, $f_0 \in \cP_2^{\cm}(\R)$, by arguing that there exists a sequence of empirical measures $\bra[\big]{f^n_0=\frac1n\sum_{i=1}^n \delta_{v_i^n(0)}}_{n \in \N}$ such that
\begin{align}
    d_2(f_0, f^n_0) \to 0 , \qquad\text{as } n\to \infty . 
\end{align}
Taking the existence of $f^n_0$ for granted, we can then follow the atoms of the initial empirical measure $f^n_0$ along the solution of the associated system of ordinary differential equations
\begin{align}
      \frac{\dx{v}_i^n}{\dx{t}} = -\frac2n \sum_{j=1}^n \sigma_e\bra*{\abs{v_i^n(t)-v_j^n(t)}}\bra*{v_i^n(t)-v_j^n(t)},
\end{align}
whose existence is guaranteed by the classical Cauchy--Lipschitz theory.
This gives rise to a family of curves $(f^n_t)_{t\in [0,T]}$ for each $n\in \N$, which are readily verified to be weak solutions to~\eqref{eq:AggregationEquation} and, by Theorem~\ref{thm:weak-curves}, also curves of maximal slope in the sense of Definition \ref{def:CurvesMaxSlope}. In particular, this sequence of solution satisfies the a priori estimate~\eqref{eq:energy-dissipation}, and they have uniformly bounded action, thus they are curves in $\AC\bra*{[0,T],(\cP_2^{\cm}(\R),d_{\cA})}$. Moreover, since convergence in $d_2$ implies $f^n_0 \to f_0$ and convergence of second order moments, we also obtain $\cE(f^n_0) \to \cE(f_0)$. 
Hence, we can conclude the proof by applying the stability statement from Theorem~\ref{thm:stability} in the limit $n\to \infty$ and conclude
\[
  0 = \liminf_{n\to \infty} \cG_T(f^n) \geq \cG_T(f)  \geq 0.
\]
Hence the limit $f$ is also a curve of maximal slope.

\medskip

Let us now turn to the construction of the approximation $f^n_0$ of the initial measure $f_0$, which consists of three steps: mollification, truncation, and approximation by particles. Let $\eps > 0$ be arbitrary. 

\medskip
\noindent
\emph{Step 1.} In the mollification step, we find some $f_{\mathrm{ac}}^\eps \in \Leb^1(\R) \cap \cP(\R)$ such that $d_2(f_0, f_{\mathrm{ac}}^\eps) < \eps/3$, which can be easily done by mollifying $f_0$ with a smooth bump function at a suitable scale $\delta=\delta(\eps)>0$. Furthermore, we note that
\begin{align}
    \int_{\R}\abs{v}^2\dx f_{\mathrm{ac}}^\eps(v)  = &
    \int_{\R}\int_{\R} \abs{v}^2 \varphi^\delta(v-w) \dx{v} \dx{f_0}(w) 
    \leq \int_{\R} \bra[\big]{2\abs{w}^2 + 2\delta^2} \dx{f_0}(w) = 4 \cE(f_0) + 2\delta^2 .
\end{align}

\medskip
\noindent
\emph{Step 2.} We will now use the fact that the second moment control on $f_{\mathrm{ac}}^\eps$, gives us uniform tightness which allows to cut off, in a quantitative fashion, its tails.  The standard tightness estimate tells us that
\begin{align}
    \int_{[-R,R]^c} \dx{f_{\mathrm{ac}}^\eps} \leq \frac{1}{R^2} \int_{[-R,R]^c} \abs{v}^2 \dx{f_{\mathrm{ac}}^\eps} \leq \frac{4 \cE(f_0) + 2\delta^2}{R^2} \, .
\end{align}
Consider now the cut off and renormalised measure $f_{\mathrm{ac},R}^\eps= f_{\mathrm{ac}}^\eps|_{[-R,R]}/\norm{f_{\mathrm{ac}}^\eps}_{\Leb^1([-R,R])}$. Using~\cite[Theorem 6.15]{V2}, we have that
\begin{align}
    d_2(f_{\mathrm{ac}}^\eps,f_{\mathrm{ac},R}^\eps) \leq & \bra*{2\int_{\R}\abs{v}^2 \abs{f_{\mathrm{ac}}^\eps-f_{\mathrm{ac},R}^\eps}\dx{v}  }^{\frac{1}{2}} \\
    \leq &\bra*{\frac{2(1- \norm{f_{\mathrm{ac}}^\eps}_{\Leb^1([-R,R])})}{\norm{f_{\mathrm{ac}}^\eps}_{\Leb^1([-R,R])}} }^{\frac12} \bra*{4 \cE(f_0) + 2\delta^2}^{\frac12} +\sqrt{2} \int_{[-R,R]^c}\abs{v}^2\dx{f_{\mathrm{ac}}^\eps} \,.
\end{align}
It is now clear that for a fixed $\eps>0$, we can choose $R=R(\eps)>0$ such that it holds that
\begin{align}
    d_2(f_{\mathrm{ac}}^\eps,f_{\mathrm{ac},R}^\eps) <\frac\eps3.
\end{align}

\medskip
\noindent
\emph{Step 3.} Finally, we use a classical result from measure theory (for example cf.~\cite[Example 8.16 (i)]{Bog07})  that empirical measures are dense in probability measures in the narrow topology. However, since $f_{\mathrm{ac},R}^\eps$ has compact support, the sequence of empiricals we construct will necessarily converge in $d_2$. Thus, we can find a measure of the form $f^n_0:=\frac1n\sum_{i=1}^n\delta_{v_i}$ for some $n=n(\eps)$ such that 
\begin{align}
    d_2(f^n_0,f_{\mathrm{ac},R}^\eps) < \frac\eps3.
\end{align}
This completes the proof of the existence of an approximating sequence of empirical measures and hence the proof.
\end{proof}

\section*{Appendix}

\subsection*{Formal derivation of the Boltzmann equation}
We present a formal derivation of the Boltzmann equation from a gain-loss argument. For the subsequent argument, it is more useful to think of the collisions in terms of the matrix $T: \R^2 \to \R^2$ given by
\[
    T=
    \begin{pmatrix}
    \frac{1- e}{2} & \frac{1+ e}{2} \\
    \frac{1+ e}{2} & \frac{1- e}{2}
    \end{pmatrix},
\]
 which maps the pre-collisional velocities to the post-collisional velocities, i.e.,
\[
\begin{pmatrix}
v'\\
v'_*
\end{pmatrix}=T\begin{pmatrix}
v\\
v_*
\end{pmatrix}.
\]
Respectively, its inverse, given by
\[
T^{-1}=
\begin{pmatrix}
\frac{1- e^{-1}}{2} & \frac{1+ e^{-1}}{2} \\
\frac{1+ e^{-1}}{2} & \frac{1- e^{-1}}{2}
\end{pmatrix} \, ,
\]
maps post-collisional velocities to pre-collisional velocities. Note that $\det T= -e$ and $\det (T^{-1})= - e^{-1}$. 

A formal derivation for the inelastic Boltzmann equation can be obtained by describing the evolution of the velocity distribution, $f_t$, using a simple gain-loss balance argument. The density at a point $v$ in velocity space is produced by all collisions of particles with `$v$' as one of their post-collisional velocities and is destroyed by all collisions with `$v$' as one of their pre-collisional velocities.

We thus split the derivation into two parts: gain and loss. We consider an  $\eps>0$ interval $\Omega_\eps=[\nu-\eps ,\nu+\eps]$ around some velocity $\nu$ and try to obtain the rate of production of density in this interval. Formally, we can integrate over the rate of production for those pre-collisional velocities $\alpha= T^{-1}_1(v,v_*)$ and $\beta= T^{-1}_2(v,v_*)$ that  produce $v$ after collision and arrive at
\begin{align}
    \bra*{\int_{{\Omega_\eps}}  \partial_t \ft(v) \dx{v}}_{\textrm{gain}} =\iint_{\R^2} \ft(\alpha) \ft(\beta ) \sigma(\abs{\alpha -\beta})  \mathds{1}_{\Omega_\eps}(v) \dx{\alpha} \dx{\beta}.
\end{align}
The function $\sigma=\sigma(|v|)$ models the frequency of the collisions, depending on the strength of the relative velocities and referred to as the collision kernel. We now make the change of variables $(\alpha,\beta) \mapsto (v,v_*)$ to obtain
\begin{align}
    \bra*{\int_{{\Omega_\eps}} \partial_t \ft(v) \dx{v}}_{\textrm{gain}} = e \iint_{\R^2} \ft(T^{-1}_1(v,v_*)) \ft(T^{-1}_2(v,v_*)) \sigma( e^{-1}\abs{v-v_*})  \mathds{1}_{\Omega_\eps} (v) \dx{v} \dx{v_*}.
\end{align}
The loss term is simpler as we obtain
\begin{align}
    \bra*{\int_{{\Omega_\eps}}  \partial_t \ft(v) \dx{v}}_{\textrm{loss}} = \iint_{\R^2} \ft(v) \ft(v_*) \sigma(\abs{v-v_*})  \mathds{1}_{\Omega_\eps} (v) \dx{v} \dx{v_*},
\end{align}
where we have integrated over the rate of destruction over all pre-collisional velocities with one of the particles having velocity $v$. Subtracting the two, dividing by $\eps$, and passing to the limit we have
the strong form as
\begin{align}
    \partial_t f_t(v) =&  e \int_{\R} \ft(T^{-1}_1(v,v_*)) \ft(T^{-1}_2(v,v_*)) \sigma( e^{-1}\abs{v-v_*})   \dx{v_*} \\ &-\int_{\R} \ft(v) \ft(v_*) \sigma(\abs{v-v_*})    \dx{v_*}.
\end{align} 
The weak form can be obtained by testing against $\varphi \in C^\infty(\R)$ as follows
\begin{align}\label{eq:weak:T}
  \begin{split}
    \skp*{\varphi, \partial_t \ft} =&e \iint_{\Rdiag} \ft(T^{-1}_1(v,v_*)) \ft(T^{-1}_2(v,v_*)) \sigma( e^{-1}\abs{v-v_*})  \varphi(v) \dx{v} \dx{v_*}\\
    & -\iint_{\Rdiag} \ft(v) \ft(v_*) \sigma(\abs{v-v_*})  \varphi(v) \dx{v} \dx{v_*} \, .
    \end{split}
\end{align}
 We would now like to bring the collision operator into a more standard form. To this end, we relabel the gain term and
change variables back to $(v,v_*)= T^{-1}(v',v'_*)$, to obtain
\begin{align}
    \skp*{\varphi, \partial_t f_t} & =e \iint_{\Rdiag} \ft(T^{-1}_1(v',v'_*)) \ft(T^{-1}_2(v',v'_*)) \sigma( e^{-1}\abs{v'-v'_*})  \varphi(v') \dx{v'} \dx{v'_*}\\& \quad -\iint_{\Rdiag} \ft(v) \ft(v_*) \sigma(\abs{v-v_*})  \varphi(v) \dx{v} \dx{v_*}\\
    &= \iint_{\Rdiag} \ft(v) \ft(v_*) \sigma( \abs{v-v_*})  \varphi(v') \dx{v} \dx{v_*}\\
    &\quad -\iint_{\Rdiag} \ft(v) \ft(v_*) \sigma(\abs{v-v_*})  \varphi(v) \dx{v} \dx{v_*} \\
    &= \iint_{\Rdiag} \ft(v) \ft(v_*) \sigma( \abs{v-v_*})  (\varphi(v') -\varphi(v)) \dx{v} \dx{v_*} =\skp*{\varphi, Q(\ft,\ft_*)} \, .
\end{align}
One can symmetrise once more by using the transformation $v \mapsto v_*$  which also induces the transformation $v' \mapsto v'_*$. Thus, one obtains
\begin{align}
    \label{eq:wf:derivation}
    \begin{split}
    \skp*{\varphi,Q(\ft,\ft_*)} &=  \frac{1}{2}\iint_{\Rdiag} \ft(v) \ft(v_*) \sigma( \abs{v-v_*})  \bra*{\varphi(v')+ \varphi(v'_*) -\varphi(v) -\varphi(v_*)} \dx{v} \dx{v_*} \, .
    \end{split}
\end{align}

\section*{Acknowledgements}
AE, RSG, and MS would like to thank José Antonio Carrillo (Oxford) for introducing them to this fascinating topic and encouraging them to work on this problem. The authors are deeply grateful to the reviewers for their valuable comments. A large part of this work was completed while all four authors were at the Hausdorff Research Institute for Mathematics (Bonn) during the Junior Trimester Program on \emph{Kinetic Theory} and while AE, RSG, and MS were at the Institut Henri Poincar\'e (Paris) during their \emph{Research in Paris} stay. The authors are grateful to both institutes for their support and hospitality. AE was supported by the Advanced Grant Nonlocal-CPD (Nonlocal PDEs for Complex Particle Dynamics: Phase Transitions, Patterns and Synchronization) of the European Research Council Executive Agency (ERC) under the European Union’s Horizon 2020 research and innovation programme (grant agreement No. 883363). A considerable part of this work was carried out while AE was a postdoc at FAU Erlangen-N\"{u}rnberg. AE gratefully acknowledge support by the German Science Foundation (DFG) through CRC TR 154  ``Mathematical Modelling, Simulation and Optimization Using the Example of Gas Networks". AS is supported by the German Research Foundation (DFG) under Germany's Excellence Strategy EXC 2044 -- 390685587,
\emph{Mathematics M\"unster: Dynamics--Geometry--Structure}.

\section*{Availability of data and materials}

Data sharing not applicable to this article as no datasets were generated or analysed during the current study.

\bibliography{references}
\bibliographystyle{abbrv}

\end{document}